\documentclass{amsart}

\title[Reflection $+$ abundant urelements bi-interprets supercompactness]{Reflection in second-order set theory with abundant urelements bi-interprets a supercompact cardinal}

\author{Joel David Hamkins}
\address[Joel David Hamkins]
{O'Hara Professor of Philosophy and Mathematics, University of Notre Dame, 100 Malloy Hall, Notre Dame, IN 46556 USA \&\ Associate Faculty Member, Professor of Logic, Faculty of Philosophy, University of Oxford, UK}
\email{jdhamkins@nd.edu}
\urladdr{http://jdh.hamkins.org}

\author{Bokai Yao}
\address[Bokai Yao]{
University of Notre Dame}
 \email{byao1@nd.edu}
\urladdr{https://philpeople.org/profiles/bokai-yao}

\thanks{The authors are grateful to the anonymous referee for extremely detailed, helpful comments. Commentary about this article can be made on the first author's blog at \href{http://jdh.hamkins.org/second-order-reflection-with-abundant-urelements}{http://jdh.hamkins.org/second-order-reflection-with-abundant-urelements}.}

\usepackage[pdfauthor={Joel David Hamkins and Bokai Yao},
    pdftitle={The strength of reflection in Kelley-Morse set theory with urelements},
    hidelinks]{hyperref}

\usepackage{latexsym,amsfonts,amsmath,amssymb,mathrsfs}
\usepackage{relsize}
\usepackage[svgnames]{xcolor}
\usepackage{wrapfig} 
\usepackage{caption} 
\DeclareCaptionStyle{compact}{font=footnotesize,format=plain,name={Fig},labelsep=period,skip=1.5ex,margin=0pt,oneside,justification=centering}
\DeclareCaptionStyle{leftside}{style=compact,margin={0pt,9pt}}
\DeclareCaptionStyle{rightside}{style=compact,margin={9pt,0pt}}
\usepackage{tikz}
\usetikzlibrary{arrows,arrows.meta,petri,topaths,positioning,shapes,shapes.misc,patterns,calc,decorations.pathreplacing,hobby}
\usepackage[shortlabels]{enumitem} 

\newtheorem{theorem}{Theorem}
\newtheorem*{theorem*}{Theorem}
\newtheorem{maintheorem}[theorem]{Main Theorem}

\newtheorem*{maintheorem*}{Main Theorem}
\newtheorem*{maintheorems*}{Main Theorems}
\newtheorem{corollary}[theorem]{Corollary}
\newtheorem*{corollary*}{Corollary}
\newtheorem*{corollaries*}{Corollaries}

\newtheorem{lemma}[theorem]{Lemma}

\theoremstyle{definition}
\newtheorem{definition}[theorem]{Definition}
\newtheorem*{definition*}{Definition}

\newtheorem*{question*}{Question}

\newtheorem*{questions*}{Questions}
\newtheorem*{mainquestion*}{Main Question} 
\newtheorem*{openquestion*}{Open Question} 
\theoremstyle{remark}

\newcommand{\QED}{\end{proof}}

\def\proclaim[#1]{{\bf #1}}
\def\BF#1.{{\bf #1.}}

\def\says#1:#2\par{\item[#1] #2\par}

\newcommand{\Los}{\L o\'s}

\newcommand{\Godel}{G\"odel}

\newcommand{\Erdos}{Erd\H{o}s}
\newcommand{\Levy}{L\'{e}vy}

\newcommand{\A}{{\mathbb A}}

\newcommand{\N}{{\mathbb N}}

\newcommand{\R}{{\mathbb R}}

\newcommand{\overbar}[1]{\mkern 3.5mu\overline{\mkern-3.5mu#1\mkern-.5mu}\mkern.5mu}
\newcommand{\barin}{\mathrel{\mkern3mu\overline{\mkern-3mu\in\mkern-1.5mu}\mkern1.5mu}}
\newcommand{\Abar}{{\overbar{A}}}

\newcommand{\Mbar}{{\overbar{M}}}
\newcommand{\Nbar}{{\overbar{N}}}
\newcommand{\Vbar}{{\overbar{V}}}
\newcommand{\Wbar}{{\overbar{W}}}

\newcommand{\id}{\mathop{\hbox{\small id}}}
\makeatletter
\newcommand{\dotminus}{\mathbin{\text{\@dotminus}}}
\newcommand{\@dotminus}{%
  \ooalign{\hidewidth\raise1ex\hbox{.}\hidewidth\cr$\m@th-$\cr}%
}
\makeatother
\newcommand{\unaryminus}{\scalebox{0.67}[1.0]{\( - \)}}

\newcommand{\of}{\subseteq}

\newcommand{\fo}{\supseteq}

\newcommand{\set}[1]{\{\,{#1}\,\}}

\newcommand{\elesub}{\prec}

\newcommand{\Con}{\mathop{{\rm Con}}}

\newcommand{\image}{\mathbin{\hbox{\tt\char'42}}}

\newcommand{\restrict}{\upharpoonright} 
\newcommand{\satisfies}{\models}

\renewcommand{\setminus}{\raise.3ex\hbox{\rotatebox{-20}{$-$}}} 
\newcommand{\union}{\cup}
\renewcommand{\emptyset}{\varnothing}

\newcommand{\Union}{\bigcup}

\newcommand{\intersect}{\cap}
\newcommand{\Intersect}{\bigcap}

\newcommand{\smalllt}{\mathrel{\mathchoice{\raise2pt\hbox{$\scriptstyle<$}}{\raise1pt\hbox{$\scriptstyle<$}}{\raise0pt\hbox{$\scriptscriptstyle<$}}{\scriptscriptstyle<}}}
\newcommand{\smallleq}{\mathrel{\mathchoice{\raise2pt\hbox{$\scriptstyle\leq$}}{\raise1pt\hbox{$\scriptstyle\leq$}}{\raise1pt\hbox{$\scriptscriptstyle\leq$}}{\scriptscriptstyle\leq}}}

\newcommand{\ltkappa}{{{\smalllt}\kappa}}

\newcommand{\ltlambda}{{{\smalllt}\lambda}}

   \def\DHLhksqrt#1#2{%
   \setbox0=\hbox{$#1\sqrt{#2\,}$}\dimen0=\ht0
   \advance\dimen0-0.2\ht0
   \setbox2=\hbox{\vrule height\ht0 depth -\dimen0}%
   {\box0\lower0.4pt\box2}}

\def\[#1]{\mathopen{\lbrack\!\lbrack}#1\mathclose{\rbrack\!\rbrack}}
\newbox\gnBoxA
\newbox\gnBoxB
\newdimen\gnCornerHgt
\setbox\gnBoxA=\hbox{\tiny$\ulcorner$}
\global\gnCornerHgt=\ht\gnBoxA
\newdimen\gnArgHgt
\def\gcode #1{%
\setbox\gnBoxA=\hbox{$#1$}%
\setbox\gnBoxB=\hbox{$\bar #1$}%
\gnArgHgt=\ht\gnBoxB%
\ifnum     \gnArgHgt<\gnCornerHgt \gnArgHgt=0pt%
\else \advance \gnArgHgt by -\gnCornerHgt%
\fi \raise\gnArgHgt\hbox{\tiny$\ulcorner$} \box\gnBoxA %
\raise\gnArgHgt\hbox{\tiny$\urcorner$}}
\newcommand{\UnderTilde}[1]{{\setbox1=\hbox{$#1$}\baselineskip=0pt\vtop{\hbox{$#1$}\hbox to\wd1{\hfil$\sim$\hfil}}}{}}
\newcommand{\Undertilde}[1]{{\setbox1=\hbox{$#1$}\baselineskip=0pt\vtop{\hbox{$#1$}\hbox to\wd1{\hfil$\scriptstyle\sim$\hfil}}}{}}
\newcommand{\undertilde}[1]{{\setbox1=\hbox{$#1$}\baselineskip=0pt\vtop{\hbox{$#1$}\hbox to\wd1{\hfil$\scriptscriptstyle\sim$\hfil}}}{}}
\newcommand{\UnderdTilde}[1]{{\setbox1=\hbox{$#1$}\baselineskip=0pt\vtop{\hbox{$#1$}\hbox to\wd1{\hfil$\approx$\hfil}}}{}}
\newcommand{\Underdtilde}[1]{{\setbox1=\hbox{$#1$}\baselineskip=0pt\vtop{\hbox{$#1$}\hbox to\wd1{\hfil\scriptsize$\approx$\hfil}}}{}}

\renewcommand{\implies}{\mathrel{\rightarrow}}

\renewcommand{\iff}{\mathrel{\leftrightarrow}}

\def\<#1>{\left\langle#1\right\rangle}

\newcommand{\QEDbox}{\Box}   

\newcommand{\TC}{\mathop{{\rm TC}}}

\newcommand{\Ord}{\mathord{{\rm Ord}}}

\newcommand{\ETR}{{\rm ETR}}

\newcommand{\ZFC}{{\rm ZFC}}
\newcommand{\ZF}{{\rm ZF}}

\newcommand{\ZFCm}{\ZFC^-}

\newcommand{\ZFCU}{{\rm ZFCU}}
\newcommand{\ZFU}{{\rm ZFU}}

\newcommand{\KM}{{\rm KM}}
\newcommand{\GB}{{\rm GB}}
\newcommand{\GBC}{{\rm GBC}}
\newcommand{\GBc}{{\rm GBc}}

\newcommand{\GBCU}{{\rm GBCU}}
\newcommand{\AAA}{{\rm AAA}}

\newcommand{\AC}{{\rm AC}}
\newcommand{\DC}{{\rm DC}}
\newcommand{\CC}{{\rm CC}}

\newcommand{\AD}{{\rm AD}}

\newcommand{\cell}[1]{\boxit{\hbox to 17pt{\strut\hfil$#1$\hfil}}}
\newcommand{\head}[2]{\lower2pt\vbox{\hbox{\strut\footnotesize\it\hskip3pt#2}\boxit{\cell#1}}}
\newcommand{\boxit}[1]{\setbox4=\hbox{\kern2pt#1\kern2pt}\hbox{\vrule\vbox{\hrule\kern2pt\box4\kern2pt\hrule}\vrule}}
\newcommand{\Col}[3]{\hbox{\vbox{\baselineskip=0pt\parskip=0pt\cell#1\cell#2\cell#3}}}
\newcommand{\tapenames}{\raise 5pt\vbox to .7in{\hbox to .8in{\it\hfill input: \strut}\vfill\hbox to
.8in{\it\hfill scratch: \strut}\vfill\hbox to .8in{\it\hfill output: \strut}}}
\newcommand{\Head}[4]{\lower2pt\vbox{\hbox to25pt{\strut\footnotesize\it\hfill#4\hfill}\boxit{\Col#1#2#3}}}
\newcommand{\Dots}{\raise 5pt\vbox to .7in{\hbox{\ $\cdots$\strut}\vfill\hbox{\ $\cdots$\strut}\vfill\hbox{\
$\cdots$\strut}}}
\usepackage[backend=bibtex,style=alphabetic,maxbibnames=15,maxcitenames=6,dateabbrev=false]{biblatex}
\renewcommand{\UrlFont}{} 
\renewbibmacro{in:}{\ifentrytype{article}{}{\printtext{\bibstring{in}\intitlepunct}}} 
\DeclareFieldFormat{url}{\UrlFont\url{#1}} 
\DeclareFieldFormat{urldate}{
  (version \thefield{urlday}\addspace%
  \mkbibmonth{\thefield{urlmonth}}\addspace%
  \thefield{urlyear}\isdot)}
\DeclareFieldFormat{eprint:arxiv}{
  \ifhyperref
    {\href{http://arxiv.org/abs/#1}{%
        arXiv\addcolon\nolinkurl{#1}}\iffieldundef{eprintclass}{}{\UrlFont{\mkbibbrackets{\thefield{eprintclass}}}}}
    {arXiv\addcolon\nolinkurl{#1}\iffieldundef{eprintclass}{}{\UrlFont{\mkbibbrackets{\thefield{eprintclass}}}}}}
\newcommand\KMU{\textup{KMU}}

\newcommand\AAvec{\smash{\vec{\A}}}
\newcommand\varin{\mathrel{\varepsilon}}
\newcommand\Vcalbar{\overline{\mathcal{V}}}
\renewcommand\Vbar{\overbar{V}}
\newcommand\ZFUvec{\textup{ZF}\vec{\textup{U}}}
\newcommand\ZFCUvec{\textup{ZFC}\vec{\textup{U}}}

\begin{document}

\begin{abstract}
After reviewing various natural bi-interpretations in urelement set theory, including second-order set theories with urelements, we explore the strength of second-order reflection in these contexts. Ultimately, we prove, second-order reflection with the abundant atom axiom is bi-interpretable and hence also equiconsistent with the existence of a supercompact cardinal. The proof relies on a reflection characterization of supercompactness, namely, a cardinal $\kappa$ is supercompact if and only if every $\Pi^1_1$ sentence true in a structure $M$ (of any size) containing $\kappa$ in a language of size less than $\kappa$ is also true in a substructure $m\elesub M$ of size less than~$\kappa$ with $m\intersect\kappa\in\kappa$.
\end{abstract}

\maketitle

\section{Set theory with urelements}

\begin{wrapfigure}[11]{o}{.36\textwidth}\vskip-1ex\hfill
\begin{tikzpicture}[yscale=1.5,scale=1.2]
\draw[fill=Orange] (0,0) -- (1,0) node[label={[scale=.7]right:Atoms}] {} -- (2,2) node[below right,scale=.7] {$V(A)$} -- (-1,2) -- (0,0);
\draw[fill=Yellow] (0,0) -- (-1,2) -- (1,2) -- cycle;
\draw (-.0625,.125) -- (1.0625,.125) node[scale=.7,label={[scale=.6]right:Sets of atoms}] {};
\draw (-.125,.25) -- (1.125,.25) node[scale=.7,label={[scale=.5]right:Sets of sets of atoms}] {};
\draw (0,1.5) node[scale=.6,align=center] {Pure\\ sets};
\draw[DarkRed,dotted,line width=2pt,line cap=round, dash pattern=on 0pt off \pgflinewidth] (0,0) to node[below,scale=.7] {$A$} (1,0);
\end{tikzpicture}
\end{wrapfigure}
Set theory was traditionally often conceived as a theory of abstract collection taking place over a class of already-existing primitive objects, the \emph{urelements} or \emph{atoms}---we use the terms interchangeably---the irreducible mathematical objects, which are not sets but out of which the sets are to be formed. Perhaps we have numbers as these primitive urelements, or perhaps geometric points or some other kind of atomic, irreducible mathematical object. With these urelement atoms in hand, we proceed to form the sets of atoms, sets of sets of atoms, and so on. The cumulative set-theoretic universe thus grows out of the atoms.

So let us consider this urelement set theory. If $A$ is the class of all urelement atoms, we denote the corresponding set-theoretic universe by $V(A)$, reserving the symbol $V$ to denote the class of pure sets, which have no urelements amongst their hereditary members; we do not intend to suggest with this notation that $V(A)$ was somehow constructed from $V$ by adding urelements $A$, although we shall explain (page \pageref{Page.rank-hierarchy}) senses in which $V(A)$ does admit a cumulative hierarchical structure. Rather, the central structure in focus is $\<V(A),\in,\A>$, a model of the theory we shall presently describe, where $V(A)$ denotes the class of all sets and atoms and $\A$ is a predicate picking out the atoms, so that $\A(a)$ holds exactly when $a$ is an urelement atom. This predicate distinguishes the urelements from the empty set $\emptyset$, which would otherwise appear like an urelement in having no elements, but amongst these $\in$-minimal objects, only $\emptyset$ is a set. The predicate $\A$ is thus definable in $\<V(A),\in>$ from the parameter $\emptyset$, and we could dispense with the predicate $\A$ by instead specifying $\emptyset$ as a constant.

\subsection{Axioms}

In the urelement context, the axiom of extensionality should be taken as the assertion that whenever two sets---two non-urelement objects---have all the same elements, then they are identical. The urelements, of course, have all the same elements as each other, none, but they are not sets and so do not constitute a counterexample to this formulation of extensionality. The other axioms we include in the theory \ZFCU\ of set theory with urelements are the axioms of  foundation, pairing, union, power set, infinity, the separation scheme, the collection scheme, and the axiom of choice, all stated in the language of urelement set theory $\set{\in,\A}$. Without the axiom of choice, we denote the theory by \ZFU.

Note that we include collection as an axiom of \ZFCU, but be careful, since not all authors do so. The replacement axiom, of course, is a consequence of collection and separation, but the converse is not true, and we shall prove (page \pageref{Page.Collection-fails}) that it is a strictly weaker theory to omit the collection axiom. The reader should take note in urelement set theory that many familiar set-theoretic arguments and results, which work as expected in \ZF\ or \ZFC\ set theory, break down in the urelement context unless one has adopted just the right principles to support them. One example is the failure we just mentioned of the implication from replacement to collection over the other axioms. Another easy example might be the fact that the rank hierarchy $V_\alpha(A)$, defined below, is not set-like when there are a proper class of urelements---the levels of the cumulative hierarchy are each a proper class, if $A$ is. A third example is that, while in \ZFC\ one can prove the first-order $\alpha$-dependent choice scheme $\alpha$-\DC\ for every ordinal $\alpha$, that is, for definable class relations, this is not true in \ZFCU, even though this theory has the axiom of choice, collection, replacement, reflection and other principles that might be deemed relevant (see page \pageref{Page.DC-failure}). The general lesson is that in urelement set theory, one should ensure that one has the right theory for  whatever argument is at hand.

In the language of pure set theory $\in$, there is a sense in which one can regard \ZFCU\ as a subtheory of \ZFC, since ultimately we have simply weakened the extensionality axiom to accommodate urelements (making sure also to retain the collection axiom). But in this language, one cannot distinguish $\emptyset$ from the urelements, and so one needs $\emptyset$ as a parameter or constant in this formulation of the theory, if one intends to distinguish it from the atoms.

\subsection{The urelement support of a set}

For any set $u$ consider the urelements of which it is formed, the atoms appearing in the transitive closure $\TC(u)$ of the set, the smallest transitive set containing $u$. The urelement \emph{support} of a set $u$, also known as the \emph{kernel} of $u$, is the set of urelements appearing in the transitive closure of $u$. For any set or class of urelements $B\of A$, we may form the class $V(B)$ consisting of all sets $u$ whose support is contained in $B$. The class of pure sets $V$ is therefore the same as $V(\emptyset)$, the sets with no urelements in their transitive closure. Since every set $u$ is supported by some set of urelements, it follows that $V(A)=\Union_{w\of A}V(w)$, where $w$ ranges over all possible sets of urelements.

For any class of urelements $B\of A$, the model $V(B)$ is a transitive class containing all pure sets and closed under pairing, union, and power set. This is therefore a model of the axioms of extensionality, foundation, pairing, union, power set, infinity, and separation, as well as the axiom of choice, provided these hold in $V(A)$. If the collection axiom holds in $V(A)$, then it will also hold in $V(B)$, since if $u\in V(B)$ and every $x\in u$ has a witness $y\in V(B)$ that $\varphi(x,y)^{V(B)}$, then we may collect these in $V(A)$, and the collecting set can be restricted to its instances in $V(B)$, which has support in $B$---more generally, $V(B)$ contains every subset of it that exists in $V(A)$. In short, $V(B)$ will be a model of \ZFCU\ whenever $V(A)$ is.\goodbreak

\subsection{The rank hierarchy}

In \ZFC\ set theory, the universe $V$ of pure sets is stratified into the familiar cumulative $V_\alpha$ hierarchy of rank, which can be defined by iterating the power set operator, starting with $V_0=\emptyset$, taking power set at successors $V_{\alpha+1}=P(V_\alpha)$, and unions at limits $V_\lambda=\Union_{\alpha<\lambda}V_\alpha$. Alternatively, one can define the rank of a set directly, with the recursive definition $\rho(y)=\sup\set{\rho(x)+1\mid x\in y}$. The rank $\rho(y)$ is the smallest ordinal for which $y\in V_{\rho(y)+1}$, and so $V_\alpha$ is the collection of sets having rank less than $\alpha$.\goodbreak

In the urelement context $V(A)$, one can attempt to iterate the power set by defining $V_0(A)=A$, taking the power set (power class) with $V_{\alpha+1}(A)=P(V_\alpha(A))$ and unions at limits $V_\lambda(A)=\Union_{\alpha<\lambda}V_\alpha(A)$. When $A$ is a set, this works completely fine. When $A$ is a proper class, however, this recursion is not set-like, because every $V_\alpha(A)$ will be a proper class. Furthermore, in the proper class case this manner of definition is problematic, because one cannot generally undertake class recursions of this form without additional axiomatic power---the principle of elementary transfinite recursion \ETR, asserting that all first-order class recursions along class well-ordered relations have solutions, is known to have consistency strength strictly stronger than \GBC, even for class recursions of length $\omega$; see  \cite{GitmanHamkins2016:OpenDeterminacyForClassGames}.

\label{Page.rank-hierarchy} 
Nevertheless, the rank recursion can be successfully undertaken in urelement set theory by defining the rank function directly on individual sets with the recursion $\rho(y)=\sup\set{\rho(x)+1\mid x\in y}$, adding that $\rho(a)=\unaryminus 1$ for the atoms. Thus, we can define the class $V_\alpha(A)$ to be the atoms and sets with rank less than $\alpha$. In this way, we achieve a rank hierarchy on $V(A)$. 
 $$V(A)=\Union_{\alpha\in\Ord}V_\alpha(A)$$
An alternative way to realize this is to define $V_\alpha(A)$ to be the union of $V_\alpha(w)$ for any \textit{set} of atoms $w\of A$, using the fact that the powerset-iteration construction works fine for $V_\alpha(w)$, stratifying the class $V(w)$, and every set in $V(A)$ is in some $V(w)$, where $w$ is the set of atoms supporting the given set. 
 $$V(A)=\Union_{\genfrac{}{}{0pt}{2}{\alpha\in\Ord}{w\of A}}V_\alpha(w)$$
Thus, the universe $V(A)$ is realized as a two-dimensional set-like cumulative hierarchy, a potentialist system that allows both for increasing height as $\alpha$ increases and increasing width as one increases the set of urelements $w\of A$ used for support. 

\subsection{Rigidity and nonrigidity}

In \ZFC\ set theory, every transitive set and class is rigid and two transitive sets are isomorphic with respect to $\in$ if and only if they are identical. This is because if $X$ and $Y$ are transitive classes of pure sets and $\pi:\<X,\in>\to\<Y,\in>$ is an $\in$-isomorphism, then necessarily $\pi(y)=\set{\pi(x)\mid x\in y}$, and so by $\in$-induction it follows that $\pi(y)=y$ and consequently $X=Y$ and $\pi=\id$. In particular, every transitive set in \ZFC\ is rigid with respect to $\in$, and the set-theoretic universe $\<V,\in>$ can have no definable automorphisms (allowing parameters).

With urelements, however, this rigidity phenomenon breaks down dramatically, and any two equinumerous classes of atoms $B$ and $C$ give rise to isomorphic corresponding transitive universes $\<V(B),\in>$ and $\<V(C),\in>$. Specifically, any bijection $\pi:B\to C$ can be recursively extended by defining $\pi(y)=\set{\pi(x)\mid x\in y}$, which will ensure the isomorphism property $x\in y\iff\pi(x)\in\pi(y)$. Similarly, any permutation $\pi:A\to A$ extends to an automorphism of the full set-theoretic universe $V(A)$ itself. When there are at least two urelements atoms, therefore, the universe $V(A)$ is not rigid. Precisely because the atoms are taken to be irreducible mathematical objects with no internal set-theoretic structure, they are set-theoretically indistinguishable, and so it becomes a fundamental part of the urelement perspective that the set-theoretic universe is homogeneous with respect to the urelement atoms. Meanwhile, perhaps one introduces urelements, such as numbers or geometric points, with an intended structure to be built on top of them; the homogeneity automorphisms would carry that structure to an alternative isomorphic copy.

\subsection{Consequences of nonrigidity for urelement set theory}

This abundance of automorphisms has various effects on the set-theoretic foundations, causing sets to work differently than might naively be expected.

One such effect, as we mentioned earlier, is that the axiom of collection does not follow from the axiom of replacement, as it does in ordinary \ZFC\ set theory. To see this, assume $V(A)$ has infinitely many urelements and let $W$ be the class of finitely supported sets, that is, the class of sets $u$ whose transitive closure has only finitely many atoms. Equivalently, $W=\Union_{\text{finite }w\of A}V(w)$, where $w$ ranges over the finite sets of atoms. We observe easily that $W$ is a supertransitive class (a transitive class containing all subsets of its members), and closed under pairing, union, and power set, and so $W$ is easily seen to satisfy extensionality, foundation, pairing, union, power set, separation, as well as the axiom of choice, assuming these hold in $V(A)$.

The class $W$ also satisfies replacement, as observed by \cite{Lévy1969:Definability-of-cardinals-numbers}. To see this, suppose that $u$ has finite support $w=\set{a_0,\ldots,a_n}$ and every $x\in u$ has some unique $y\in W$ for which $W$ satisfies $\varphi(x,y)$. Note that every $x\in u$ is also in $V(w)$. What we claim is that every $x\in u$ has the corresponding $y$ also in $V(w)$, for if $y$ has some other atoms in its transitive closure, not amongst $a_0,\ldots,a_n$, then we could apply an automorphism $\pi$ to $A$ that fixes each $a_0,\ldots,a_n$ but moves these other atoms to further distinct atoms not appearing in the transitive closure of $y$, causing $\pi(y)\neq y$. But since $\pi$ would be an automorphism of $W$ fixing $u$, this other set $\pi(y)$ would also fulfill the condition $\varphi(x,\pi(y))$ in $W$, contrary to uniqueness. So every witness $y$ must be in $V(a_0,\ldots,a_n)$. Therefore the image set of all such witnesses $\set{y\mid \exists x\in u\,\varphi^W(x,y)}$, which is a set by the replacement axiom in $V(A)$, would exist in $V(a_0,\ldots,a_n)$, fulfilling the replacement axiom of $W$.

\label{Page.Collection-fails}%
Meanwhile, the bad news is that the class of finitely supported sets $W$ does not fulfill the axiom of collection. Namely, for every natural number $n$, there is a set consisting of $n$ urelements, but there is no set having such witness sets for every $n\in\omega$, since every set in $W$ has only finite support amongst the atoms; this is a violation of the collection axiom.

The situation at bottom is extremely similar to the analogous failures of collection described in \cite{GitmanHamkinsJohnstone2016:WhatIsTheTheoryZFC-Powerset?} for the naive axiomatizations of set theory without power set, which use merely the replacement axiom rather than collection plus separation. The main fact, first observed by \cite{Zarach1996:ReplacmentDoesNotImplyCollection} and extended by \cite{GitmanHamkinsJohnstone2016:WhatIsTheTheoryZFC-Powerset?}, is that all kinds of things go wrong when one uses the naive theory for \ZFC\ without power set---the collection axiom can fail; the cardinal $\omega_1$ can be singular, even when the axiom of choice holds; there can be sets of reals of every size $\aleph_n$, but no set of reals of size $\aleph_\omega$; the \Los\ theorem can fail badly, even when choice holds; the \Los\ theorem can fail even in the case of a measurable cardinal $\kappa$; the class of $\Sigma_1$ definable sets can fail to be closed under bounded quantifiers. Meanwhile, the research literature unfortunately has numerous instances (identified in \cite{GitmanHamkinsJohnstone2016:WhatIsTheTheoryZFC-Powerset?}) where researchers provide the naive axiomatization of $\ZFC^-$, but then presume that the various problematic features do not occur, even though this is not provable in their theory. In light of these failures, the main conclusion to be drawn is that the naive axiomatization is simply a mistake---the correct axiomatization of ZFC without power set uses collection and separation, rather than mere replacement. This is the theory that holds in the natural instances and applications of this theory, such as the use in $H_{\kappa^+}$, and to our knowledge all of the erroneous uses and applications of the theory are easily and completely corrected simply by using this stronger version of the theory.

The situation in urelement set theory is analogous. One should include the collection and separation axioms in \ZFCU, not just replacement, or else face a similar phenomenon of unexpected, undesirable effects.

Unfortunately, the collection/replacement issue is not the end of the story, for there are further unwelcome situations that are possible even when one does include the collection axiom.\footnote{The second author \cite{Yao2023:Dissertation} identifies the separations of a hierarchy of natural principles over the weaker version of ZFCU formulated with only replacement.} To see this, assume we have uncountably many atoms and consider the class $\Wbar$ of all (atoms and) countably supported sets. That is, let $\Wbar=\Union_{\text{countable }w\of A}V(w)$. This class $\Wbar$ will satisfy extensionality, foundation, pairing, union, power set, separation, choice, infinity, and replacement. Furthermore, this class $\Wbar$ satisfies the collection axiom. To see this, suppose that $u\in \Wbar$ and for every $x\in u$ there is $y\in \Wbar$ for which $\varphi(x,y)$ holds in $\Wbar$. Let $w$ be the support of $u$ in the atoms, and extend this set of atoms $w\of \bar w$ by adding countably many additional atoms. Since every individual $y\in \Wbar$ has countable support in the atoms, there is an automorphism $\pi:V(A)\to V(A)$ fixing every $a\in w$ and for which $\pi(y)\in V(\bar w)$. That is, whichever new atoms $y$ used outside $w$, we can map them into $\bar w$ and thereby place $\pi(y)$ in $V(\bar w)$. Thus, the witnesses $y$ can always be found inside $V(\bar w)$, and consequently we can realize the collection axiom inside some sufficiently large $V_\alpha(\bar w)$.

Meanwhile, the model $\Wbar$ has a strange property: the class of all urelements $A$ is a proper class, but every set of atoms is countable. It follows that ZFCU does not prove that any set can be mapped injectively to any proper class. Moreover, although every countable sequence of distinct atoms can be extended further, and every chain of such sequences has its union as an upper bound, and $\omega_1$ exists, yet there is no $\omega_1$-sequence of distinct atoms in $\Wbar$. \label{Page.DC-failure}%
This is a violation of the $\omega_1$-dependent choice scheme $\omega_1$-\DC, asserting that if $T\of V(A)^{<\omega_1}$ is a definable class tree of countable sequences, such that every sequence in $T$ has a proper extension in $T$ and every countable chain in $T$ has an upper bound in $T$, then $T$ admits a branch of length~$\omega_1$.

This situation is counterintuitive because the $\omega_1$-dependent choice scheme is provable in \ZFC\ set theory---perhaps we usually think of it as an elementary consequence of the axiom of choice and collection in the \ZFC\ context. And yet, it is not provable in \ZFCU, even though this theory has the axiom of choice, collection, replacement, and (we shall show) the reflection principle. One can similarly make models where the failure of $\alpha$-\DC\ occurs first at higher ordinals $\alpha$.

Let us exhibit another instance of strange behavior in urelement set theory. Suppose that we have a model of \ZFCU\ with the urelements forming a set of size $\omega_1$. Let us split the $\omega_1$ many urelements into two disjoint sets $A\sqcup B$, both of size $\omega_1$. Now define $Y$ as the class of sets $u$ in $V(A\union B)$ whose support has at most countable intersection with $B$. So every set of urelements in $Y$ is contained in $A\union w$ for some countable subset $w\of B$. Meanwhile, we can prove that $Y$ is a model of \ZFCU\ just as with $\Wbar$ above. The curious thing to notice about this model $Y$ is that the urelements do not form a set---they are a proper class---and $A$ is a set of urelements in $Y$ of size $\omega_1$, but all sets of urelements disjoint from $A$ are countable. So we cannot find a duplicate of $A$ in the urelements disjoint from $A$.

One way to reintroduce a measure of rigidity into urelement set theory is to provide a fixed enumeration of the atoms by a class of pure sets. Or simply to well-order the atoms. Consider for example the structure $\<V(A),\in,<>$, where $<$ is a well-ordering of the atoms, not necessarily set-like. This universe is definably rigid, because if $\pi:V(A)\to V(A)$ is a definable automorphism, then in respecting~$<$ it must consequently fix every atom, since there can be no least element moved, and since also $\pi(y)=\set{\pi(x)\mid x\in y}$ for the sets, it follows by $\in$-induction that $\pi(y)=y$ and so $\pi$ is the identity. In this expanded signature, the class $W$ of (atoms and) finitely supported sets no longer satisfies replacement, when there are infinitely many atoms, since for each $n$ there is a unique atom $a_n$ that is $n$th in the order $<$, but the set of all these atoms is not finitely supported. Similarly, the class $\Wbar$ of (atoms and) countably supported sets does not satisfy replacement in the expanded language, when there are uncountably many atoms, because for each countable ordinal $\alpha$ there is an $\alpha$th atom, but the set of these is not countably supported.

An alternative method, which we shall employ in the next section, is to introduce a predicate $\AAvec$ for explicitly enumerating the urelements by a class of pure sets. The predicate $\AAvec(i,a_i)$ will associate each element $i\in I$ in some class $I$ of pure sets with a distinct urelement $a_i$. For example, perhaps we shall have exactly $\omega$ many urelements, or $\R$ many, or $\Ord$ many, or $V$ many urelements.\goodbreak

\section{Interpreting urelement set theory in \ZFC}\label{Section.Interpreting-urelements-in-ZFC}

Let us describe next how to interpret urelement set theory inside the urelement-free set theories \ZF\ and \ZFC. The fact that this is possible, and furthermore the fact that it is possible in a way that achieves strong axioms for the interpreted urelement set theories, such as ``there are a proper class of atoms'' or even ``the atoms can be well-ordered in type $\Ord$,'' tends to support the view that the foundational roles sought for the urelement theories can also be undertaken ultimately without urelements by interpreting these urelement theories inside \ZFC. On this view, one will not identify any new mathematical structure or relation in these urelement set theories that we cannot in principle find already in \ZFC.

Assume $V$ is a set-theoretic universe satisfying \ZFC, and suppose $A$ is any class of sets in $V$. (In the \ZFC\ context, we assume $A$ is a definable class, but we shall later move to the \GBC\ context, where we drop the definability requirement.) Inside $V$ we shall presently define a certain model of urelement set theory $V\[A]$, a model of $\ZFCU$ in which the urelements will be indexed by the class $A$, whose pure sets, furthermore, will be an exact copy of $V$.

We begin by making a copy of $A$ with $\bar A=\set{0}\times A$, representing every object $a\in A$ by its copy $\bar a=\<0,a>$. These will be the urelements of $V\[A]$, and we accordingly place $\bar A\of V\[A]$. Next, we simply close $V\[A]$ under the operation: if $y\of V\[A]$, then $\bar y=\<1,y>\in V\[A]$. That is, $V\[A]$ is the smallest class of sets containing $\bar A$ and closed under that operation. One may undertake this as a set-like transfinite recursion in \ZF\ by forming a hierarchy $V_\alpha\[A\intersect V_\alpha]$, gradually adding atoms and closing under the operation as the recursion progresses. The central idea is that $\bar y=\<1,y>$ will be the set in $V\[A]$ whose elements are the (actual) elements of $y$. Thus, we define the membership relation as $x\barin \bar y$ if and only if $x\in y$, where $\bar y=\<1,y>$. The relation $\barin$ is therefore set-like and wellfounded in $V$, because $x\barin \bar y$ requires that $x$ has lower rank than $\bar y$. In short, the definition is that every object in $V\[A]$ will be an ordered pair, with the urelements having the form $\bar a=\<0,a>$ for $a\in A$ and fulfilling $\A(\bar a)$ and the sets having the form $\bar y=\<1,y>$ for some $y\of V\[A]$, with the $\barin$-elements of $\bar y$ being simply the actual members of $y$.

The claim is that $\<V\[A],\barin,\A>$ is a model of \ZFCU. Extensionality holds because the urelements have no $\barin$-members and if $\bar y=\<1,y>$ and $\bar z=\<1,z>$ have the same $\barin$-elements, then necessarily $y=z$ and consequently $\bar y=\bar z$. Foundation holds because $\barin$ is wellfounded---the $\barin$-elements of a set have lower rank in $V$. It is easy to show that $V\[A]$ has the union, pairing and power set axioms simply by constructing the appropriate representing set in each case. For example, if $u,v\in V\[A]$, then $\<1,\set{u,v}>$ will represent the unordered pair $\set{u,v}$ in $V\[A]$. We get the separation, replacement, and collection axioms because every subset $y\of V\[A]$ that is realized in $V$ is represented by the set $\bar y\in V\[A]$. Similarly, the axiom of choice holds in $V\[A]$, if it holds in $V$, because we can form a suitable set $y$ of choices in $V$, and then consider the corresponding set $\bar y$ in $V\[A]$.

Every set $u$ in $V$ will be represented by a corresponding set $\check u=\<1,\set{\check v\mid v\in u}>$ in $V\[A]$, for it is easy to see that $v\in u\iff \check v\barin\check u$ and consequently $\<V,\in>$ is isomorphic to $\<\smash{\hat V},\barin>$, where $\hat V=\set{\check u\mid u\in V}$, which is a transitive class in $V\[A]$. Furthermore, using the fact the rank of a set $\bar y$ in $V\[A]$ is bounded by the rank of $\bar y$ in $V$, it follows that $\hat V$ consists precisely of the pure sets of $V\[A]$. In particular, the ordinals of $V\[A]$ are precisely the $\check\alpha$ for the ordinals $\alpha$ of $V$.

Let us expand the urelement language somewhat by introducing the urelement enumeration predicate $\AAvec$, which holds exactly in the instances $\AAvec(\check a,\bar a)$ when $a\in A$. That is, $\AAvec$ is the graph of the function $\check a\mapsto\bar a$, for $a\in A$. This predicate identifies the urelements---they are exactly the range of the map $\AAvec$---and so the urelement predicate $\A$ is definable from $\AAvec$. But $\AAvec$ also provides the enumeration of the urelements by the elements of $A$, well, by the copy of $A$ in the pure sets of $V\[A]$, that is, by the class $\hat A=\set{\check a\mid a\in A}$. This enumeration $\AAvec$ will not generally be definable from $\A$ alone, since permutations of the urelements will lead to automorphisms of $V\[A]$ that preserve $\A$ but not $\AAvec$. We interpret the phrase, ``there are $A$ many atoms'' as the assertion that $\AAvec$ is a bijection between $A$ and the class of urelements. The theories $\ZFUvec$ and $\ZFCUvec$ are the analogues of $\ZFU$ and $\ZFCU$ in the language with the atom-enumeration predicate $\AAvec$.

\begin{theorem}\label{Theorem.V-bi-interpretable-V[A]-AAvec}
For any definable class $A$, the \ZF\ set-theoretic universe $\<V,\in>$ is bi-interpretable with the urelement set-theoretic universe $\<V\[A],\barin,\AAvec>$, which is a model of $\ZFUvec+``$there are $A$ many urelements.''
\end{theorem}

\begin{proof}
We have already explained how $\<V\[A],\barin,\AAvec>$ is defined in $V$, and this is half of the interpretation. For the converse interpretation, we can interpret the structure $V$ in $\<V\[A],\barin>$ as the pure sets, since the map $u\mapsto\check u$ is an isomorphism of $\<V,\in>$ with the class of pure sets in $V\[A]$. So $V$ can see precisely how it is copied into $V\[A]$. Conversely, using the $\AAvec$ enumeration, the structure $\<V\[A],\barin,\AAvec>$ can see precisely how it is constructed from the pure sets by the interpretation we have provided. So this is a bi-interpretation.
\end{proof}

In light of these various bi-interpretations, we are justified in viewing these interpreted models of urelement set theory from two perspectives. On the one hand, $V\[A]$ is a definable class in the ground model $V$, defined there as an interpreted class structure. On the other hand, since $V$ is isomorphic to the pure sets of $V\[A]$, we can view the model instead as an extension $V(\Abar)$ of the original universe $V$, obtained by adjoining a new class $\Abar$ of urelements to $V$, enumerated by the class $A$. The bi-interpretation result explains exactly how these two perspectives are equivalent accounts of the same phenomenon.

Let us next record a further observation about the nature of the interpreted urelement structure.

\begin{theorem}
Assume \ZF\ in $V$ and $A$ is a definable class. Then every set in $V\[A]$ is equinumerous with a pure set there. In particular, if the axiom of choice holds in $V$, then it also holds in $V\[A]$.
\end{theorem}

\begin{proof}
Every set $\bar y=\<1,y>$ created in $V\[A]$ is equinumerous there with the pure set $\check y$. So if \AC\ holds in $V$, then since every pure set in $V\[A]$ can be well-ordered there, and these orders can be transferred to the sets $\bar y$ in $V\[A]$, establishing \AC\ in $V\[A]$.
\end{proof}

The hypothesis that every set is equinumerous with a pure set is philosophically significant with regard to structuralism, because it follows that every mathematical structure is isomorphic to a structure in the pure sets. A structuralist in the philosophy of mathematics could therefore take this as evidence that urelements are not needed in such a situation, since one can already realize all possible mathematical structure in the universe of pure sets. Alternatively, since one mathematician's modus ponens is another's modus tollens, the observation is evidence that in order for urelements to be helpful to us in the foundations of mathematics, we should want or expect to have sets of urelements that are not equinumerous with any pure set. In particular, this would require a denial of the axiom of choice, since every well order is isomorphic to an ordinal, which is a pure set. Do we expect that there are collections of irreducible atomic mathematical objects that cannot in principle be placed into bijection with any pure set? That seems strange and wonderful, but the puzzle is that such kind of urelements would be very different from the traditional urelement proposals, such as numbers and geometric points, both of which find equinumerosity with pure sets. So the philosophical task for the urelement supporters is to explain why we would need a foundational theory to accommodate strange and bizarre set of urelements, if the intended sets of urelements are tame.

\begin{corollary}\label{Corollary.ZFC-ZFCU+A-atoms-bi-interpretable}
The theories \ZFC\ and $\ZFCUvec+``$there are $A$ many atoms'' are bi-interpretable, where $A$ is any \ZF-definable class. Similarly with $\ZF$ and $\ZFUvec+``$there are $A$ many atoms.''
\end{corollary}

What we mean is that the assertion ``there are $A$ many atoms'' is expressed in the language with the urelement enumeration predicate $\AAvec$ by asserting that this predicate enumerates the urelements using the class $A$ as it is defined in the class of pure sets of the $\ZFCUvec$ model. In particular, the following theories are bi-interpretable:
\begin{enumerate}
  \item \ZFC
  \item $\ZFCUvec+``$there are $\omega$ many atoms''
  \item $\ZFCUvec+``$there are $\R$ many atoms''
  \item $\ZFCUvec+``$there are $\aleph_{\omega^2+5}$ many atoms''
  \item $\ZFCUvec+``$there are $\Ord$ many atoms''
  \item $\ZFCUvec+``$there are $V$ many atoms''
\end{enumerate}
where we express the theories in the language with the urelement enumeration predicate $\AAvec$. The corresponding theories with \ZFCU\ in place of $\ZFCUvec$ also are bi-interpretable with each other and \ZFC, and all the above theories, provided one allows a parameter---simply use a set parameter to enumerate the atoms, as explained in theorem \ref{Theorem.Set-urelements-bi-interpretation}.

\begin{proof}[Proof of corollary \ref{Corollary.ZFC-ZFCU+A-atoms-bi-interpretable}]
The interpretation definitions used in the proof of theorem \ref{Theorem.V-bi-interpretable-V[A]-AAvec} are uniform across models, and therefore provide bi-interpretations of these theories, not just model-by-model. The theories $\ZFCU$ and $\ZFCUvec$, in the case the urelements form a set, are bi-interpretable with parameters, since one need only fix the set-sized enumeration of the urelements.
\end{proof}

It follows that $\ZFCUvec$ is not a tight theory in the sense of \cite{Enayat2016:Variations-on-a-Visserian-theme, FreireHamkins2021:Bi-interpretation-in-weak-set-theories}, since it admits these distinct bi-interpretable extensions. A theory is \emph{tight} when any extensions of it are bi-interpretable if and only if they are identical.

\begin{corollary}
$\ZFCUvec$ is not tight.
\end{corollary}

If the class of urelements is a set, then we can dispense with the need for the urelement predicates $\A$ and $\AAvec$ in the interpretation, since the enumeration will in effect be one of the sets available in $V\[A]$, which we can simply use as a parameter.

\begin{theorem}\label{Theorem.Set-urelements-bi-interpretation}
 Working in the \ZF\ set-theoretic universe $V$, if $A$ is a set, then the following class structures are bi-interpretable, allowing parameters:
 \begin{enumerate}
   \item $\<V,\in>$
   \item $\<V\[A],\barin>$
   \item $\<V\[A],\barin,\A>$
   \item $\<V\[A],\barin,\AAvec>$
 \end{enumerate}
\end{theorem}

\begin{proof}
If $A$ is a set in $V$, then the predicates $\A$ and $\AAvec$ both amount to sets in $V\[A]$, and so the latter three structures are bi-interpretable with those parameters. And we know $\<V,\in>$ is bi-interpretable with $\<V\[A],\barin,\AAvec>$ by theorem \ref{Theorem.V-bi-interpretable-V[A]-AAvec}.
\end{proof}

\begin{corollary}The following theories are parametrically bi-interpretable.
 \begin{enumerate}
     \item \ZFC
     \item $\ZFCU+``$the urelements form a set''
     \item $\ZFCUvec+$``the urelements form a set.''
 \end{enumerate}
 Similarly, the following theories are also parametrically bi-interpretable:
 \begin{enumerate}[resume]
     \item \ZF
     \item $\ZFU+$``the urelements are bijective with a pure set''
 \end{enumerate} 
\end{corollary}

\begin{proof}
If the urelements form a set, then by the axiom of choice they are bijective with some pure set, and the property of being such a pure-set-enumeration of the urelements is definable. Any such parameter will enable the bi-interpretations of theorem \ref{Theorem.Set-urelements-bi-interpretation}. That is, the theory can define a nonempty set of parameters, any one of which will work as a parameter for the bi-interpretation, and this is what it means to be parametrically bi-interpretable.
\end{proof}

If $A$ is a proper class, however, then it will turn out that $\<V,\in>$ is not bi-interpretable with $\<V\[A],\barin,\A>$, even with parameters, although these structures are mutually interpretable and \emph{semi-bi-interpretable}, meaning that $V$ can identify its copy inside $V\[A]$.\goodbreak

\begin{theorem}
 The following theories are mutually interpretable, but no two of them are bi-interpretable, even allowing parameters.
   \begin{enumerate}
     \item \ZFC
     \item \ZFCU
     \item $\ZFCU+``$the urelements form a proper class''
   \end{enumerate}
Meanwhile, \ZFC\ is semi-bi-interpretable with these latter theories.
\end{theorem}

\begin{proof}
The urelement theory in case (3) is expressed using the urelement predicate $\A$. We have already provided the mutual interpretations, since from any model $V\satisfies\ZFC$ we can interpret the model $V\[A]$ for any class $A$, and conversely in any model of \ZFCU\ the class of pure sets is a model of \ZFC. This mutual interpretation is actually a semi-bi-interpretation, since in any case, $V$ can see how it is copied into the pure sets of $V\[A]$ by the map $u\mapsto\check u$. What is not generally possible for $V\[A]$, with only the urelement predicate $\A$ and not urelement enumeration predicate $\AAvec$, is to define exactly how it arises from its pure sets.

Let us prove first that no model of theory (3) is bi-interpretable with a model of \ZFC, even allowing parameters in the interpretations. Suppose toward contradiction that a model $M$ of \ZFCU\ with a proper class of urelements is bi-interpretable with a model $N$ of \ZFC. Thus, $N$ is parametrically definably interpreted in $M$, and there is a copy $\Mbar$ of $M$ definably interpreted in $N$, such that $M$ can parametrically define an isomorphism $\tau:M\cong\Mbar$, and so there is a copy $\Nbar$ of $N$ definably interpreted in $\Mbar$, such that $N$ can parametrically define an isomorphism $\sigma:N\cong\Nbar$.

Since $M$ is a model of \ZFCU\ with a proper class of urelements, we can define an automorphism $\pi:M\cong M$ that fixes all the parameters used in defining the interpretations and the maps $\tau$ and $\sigma$, but which moves some object $u$ to $\pi(u)\neq u$. Since $\pi$ fixes the parameters used in the interpretation, it follows that $\pi$ fixes $N$, $\Mbar$, $\Nbar$, $\tau$, and $\sigma$ as definable classes in $M$.

Since $\pi$ fixes $N$ as a class, and it fixes the interpretations of the relations of $N$ as classes, it follows that $\pi\restrict N$ is an automorphism of $N$ as a \ZFC\ model. Furthermore, because $N$ can define $\Mbar$, which can define its version of $\pi$ on $\Nbar$, which $N$ can pull back from $\Nbar$ to $N$, it follows that $\pi\restrict N$ would be a definable automorphism of $N$ in $N$. But \ZFC\ is definably rigid, and so $\pi$ must fix the elements of $N$ pointwise. Consequently, it also fixes the elements of $\Mbar$ pointwise. But now, since $\pi$ also fixes $\tau$ as a class, it follows that if $\tau$ takes $u$ to $\bar u$ in $\Mbar$, then it must take $\pi(u)$ to $\pi(\bar u)$, which is the same as $\bar u$, since this is in $\Mbar$. Thus, $\tau$ takes two different objects $u\neq\pi(u)$ of $M$ to the same object $\bar u$ in $\Mbar$, contradicting the assumption that it was an isomorphism.

To prove the theorem, now, suppose that \ZFC\ were bi-interpretable with one of the other theories, allowing parameters. Since every model of \ZFCU\ with a proper class of atoms is a model of both theories (2) and (3), it would meant that every model of theory (3) was parametrically bi-interpretable with a model of \ZFC, which we just proved is impossible. So \ZFC\ is not parametrically bi-interpretable with either theory (2) or theory (3). 

It follows also that theories (2) and (3) are not parametrically bi-interpretable, for every model of \ZFC\ is also a model of \ZFCU, and we've argued that no such model is parametrically bi-interpretable with a model of theory (3). 
%
\end{proof}

The proof relied on the fact that \ZFC\ models are definably rigid, but models of \ZFCU\ with a proper class of urelements have lots of parametrically definable automorphisms, fixing any desired finite list of parameters. In this situation, there can be no bi-interpretation of models. The basic lesson is that there is a fundamental difference for interpretative power between the theories $\ZFCU$ and $\ZFCUvec$, between having the urelements merely as a class $\A$ versus having them as an explicit enumeration $\AAvec$ by a class of pure sets. One can expect to achieve bi-interpretation of urelement set theory with urelement-free set theory only in the latter case, and otherwise it is the urelement-free theory with the stronger interpretative power, since $V$ sees how it is copied to the pure sets of $V\[A]$, but $\<V\[A],\in,\A>$ cannot in general see how it arises from its pure sets.\goodbreak

\section{Reflection}

\begin{wrapfigure}{r}{.2\textwidth}\vskip-2ex\hfill
\begin{tikzpicture}
\draw[fill=Yellow!40] (0,0) -- (-1,4) to node[below,scale=.8] {$V$} (1,4) -- cycle;
\draw[fill=Yellow] (0,0) -- (-.5,2)  to node[below,scale=.7] {$V_\lambda$} (.5,2) node[right,scale=.7] {$\lambda$} -- cycle;
\end{tikzpicture}
\end{wrapfigure}
Let us recall the first-order reflection phenomenon in set theory. The \Levy-Montague reflection theorem of \ZF\ set theory asserts that for any formula $\varphi(x)$, there is an ordinal $\lambda$ such that $\varphi$ is absolute between $V_\lambda$ and $V$. A somewhat more refined version asserts of every meta-theoretically finite number $n$ that there is an ordinal $\lambda$ such that all $\Sigma_n$ formulas are absolute between $V_\lambda$ and $V$. Indeed, there is a definable closed unbounded class $C$ of such ordinals $\lambda$, the club of \emph{$\Sigma_n$-correct} cardinals. This is a theorem scheme, with a separate statement for each number~$n$. The reflection theorem is equivalent over the rest of \ZFC\ to the collection axiom (hence to replacement over Zermelo set theory), since the existence of sufficiently large reflecting ordinals $\lambda$ provides $V_\lambda$ as a suitable collecting set for the formula asserting that indeed the witnesses exist.

In the urelement context, when there are a proper class of urelements, then we should not generally expect to reflect truths from $V(A)$ specifically to the class $V_\lambda(A)$, since there will be sets of urelements $w\of A$ in $V(A)$ of cardinality larger than $\lambda$, which having rank $1$ will be in $V_\lambda(A)$, but such $w$ are not equinumerous with any ordinal in $V_\lambda(A)$. This would be a violation of reflection between $V_\lambda(A)$ and $V(A)$ for the statement $\varphi=$``every set of urelements is bijective with an ordinal.'' What we should want to do with reflection, rather than reflecting specifically to some rank-initial-segment class $V_\lambda(A)$, is to reflect truth to some transitive set.

\begin{definition}
The first-order reflection principle is the scheme of assertions that for every set $p$ and every particular formula $\varphi$ in the first-order language of urelement set theory there is a transitive set $v$ containing $p$ such that $\varphi$ is absolute between $v$ and $V(A)$.
\end{definition}

\begin{wrapfigure}[11]{r}{.34\textwidth}\vskip-3ex\hfill
\begin{tikzpicture}
\draw[fill=Orange!50] (0,0) -- (-1,4) -- (2,4) node[below right,scale=.8,align=center] {$V(A)$} -- (1,0) -- cycle;
\draw[fill=Yellow!50] (0,0) -- (-1,4) to node[below,scale=.8] {$V$} (1,4) -- cycle;
\draw (0,3) node[scale=.6,align=center] {Pure\\ sets};
\draw[fill=Yellow] (0,0) -- (-.5,2) node[left,scale=.7] {$\alpha$} to node[below,scale=.7] {$V_\alpha$} (.5,2) -- cycle;
\draw[fill=Orange] (0,0) -- (.5,2) -- (1.5,2) node[below right,scale=.7] {$V_\alpha(A)$} -- (1,0) -- cycle;
\draw[DarkRed,dotted,line width=2pt,line cap=round, dash pattern=on 0pt off \pgflinewidth] (0,0) to node[below,scale=.6,align=center] {Atoms $A$} (1,0);
\end{tikzpicture}
\end{wrapfigure}
If there is merely a set of urelements, then we can stratify the full universe $V(A)$ by the rank hierarchy $V_\alpha(A)$, and each of these is a set. This provides a set-like continuous transfinite tower $V_\alpha(A)$, and this is enough to run the usual \ZFC\ proof of the reflection theorem, where we shall be able to reflect any $\varphi$ from $V(A)$ to some $v=V_\alpha(A)$. Similarly, in a model of $\ZFCUvec+``\Ord$ many atoms'', we can similarly stratify the universe in a set-like hierarchy, using \mbox{$V_{\alpha+1}(\AAvec\restrict\alpha)$}, where $\AAvec\restrict \alpha$ restricts the urelement enumeration to the first $\alpha$ many elements, and taking unions at limit ordinals for continuity. This continuous cumulative hierarchy is again enough to prove the reflection theorem by the usual argument, reflecting any $\varphi$ from $V(A)$ to some $v=\Union_{\alpha<\lambda}V_{\alpha+1}(\AAvec\restrict\alpha)$.

But in fact, we should like to prove that we get the first-order reflection theorem in \ZFCU\ generally, making no assumption on the number or kind of urelements. Let's begin by establishing the following duplication and homogeneity lemmas, showing that over some fixed set of urelements, all additional sets of urelements of the same cardinality look exactly alike.

\begin{lemma}[urelement duplication]\label{Lemma.Duplication}
Assume \ZFU\ plus every set of urelements is well-orderable. Then there is a set $u$, such that every set of urelements $w$ disjoint from $u$ can be duplicated, that is, $w$ is equinumerous with some set of urelements $w'$ disjoint from both $u$ and $w$.
\end{lemma}

\begin{proof}
Assume \ZFU\ plus every set of urelements is wellorderable, and let $u$ be any set of urelements for which the supremum of the sizes of the sets of urelements $w$ disjoint from $u$ is as small as possible. (Perhaps this supremum is unbounded in the cardinals, and that will be fine.) Notice that if every cardinal $\gamma$ in some set of cardinals is realized by a set of urelements disjoint from $u$, then we can collect such sets together and thereby realize a cardinal at least as great as each of them. It follows by the choice of $u$ that if $w$ is a set of urelements disjoint from $u$, then we must be able to realize this cardinal again disjoint from $u\union w$, for otherwise $u\union w$ would have served as a better choice than $u$. Thus, we have fulled the duplication property.
\end{proof}\goodbreak

The urelement duplication property will not necessarily hold over any set $u$, since in the model $Y$ discussed earlier, the uncountable set $A$ of urelements cannot be duplicated to a disjoint set of urelements of the same size, since all the remaining sets of urelements are countable subsets of $B$, and so in that model $Y$, we do not achieve duplication over $\emptyset$. In the model $Y$, however, we do achieve the duplication property over the set $u=A$, since any set $w$ disjoint from this $u$ can be duplicated to another such set $w'$ disjoint from $u\union w$. Moreover, the second author \cite{Yao2023:Dissertation} proves that it is  consistent with ZFU + $\omega$-DC that the duplication property holds over no set of urelements.

\begin{lemma}[urelement duplication$\to$homogeneity]\label{Lemma.Homogeneity}
Assume \ZFU. If all sets of urelements disjoint from $u$ can be duplicated, then any two equinumerous such sets $w$ and $w'$ are automorphic over $u$---there is an automorphism of the universe $V(A)$ swapping $w$ with $w'$ and fixing every element of $u$.
\end{lemma}

\begin{proof}
Assume \ZFU\ and sets of urelements disjoint from $u$ can be duplicated. Suppose that $w$ and $w'$ are equinumerous and disjoint from $u$. If $w$ and $w'$ are disjoint, then the equinumerosity produces a permutation $\pi$ of the union $v=w\union w'$. This permutation extends to a permutation $\pi:A\to A$ of the class of all urelements by fixing all other urelements---in particular, fixing every element of $u$---and this permutation generates an automorphism $\pi:V(A)\to V(A)$ showing that $w$ and $w'$ are automorphic over $u$, as desired.

But perhaps $w$ and $w'$ are not disjoint. By the duplication property, however, there is another set disjoint from $u\union w\union w'$  that is equinumerous with $w\union w'$, and in particular, there is a set $w''$ disjoint from $u\union w\union w'$ equinumerous with $w$. So $w$ is automorphic with $w''$ over $u$ by the previous paragraph, and also $w''$ is automorphic to $w'$. By composing these automorphisms, we see that $w$ is automorphic to $w'$.
\end{proof}

We cannot in general dispense with the fixed set $u$ in the homogeneity claim of lemma \ref{Lemma.Homogeneity}, for although as we have observed $V(B)$ is isomorphic to $V(C)$ whenever $B$ and $C$ are equinumerous classes of urelements, such an isomorphism is not necessarily realizable by an automorphism of $V(A)$. Indeed, it is not true in general in \ZFCU\ that whenever two sets of urelements are equinumerous, then they are automorphic, since perhaps there are exactly $\omega_1$ many urelements and $w$ is the set of all urelements, but $w'$ has all but one of them. These sets are equinumerous by a bijection in $V(w)$ and $V(w)$ is isomorphic to $V(w')$, but the sets are not automorphic by any automorphism of $V(w)$.\goodbreak

\begin{theorem}\label{Theorem.ZFCU-proves-first-order-reflection}\
 \begin{enumerate}
   \item \ZFCU\ proves the $\omega$-dependent choice scheme $\omega$-\DC. \footnote{See \cite{Schlutzenberg2021.MO387471:Does-the-axiom-scheme-of-collection-imply-DC?} and \cite{Yao2023:Dissertation} for alternative proofs of this result.}
   \item \ZFCU\ proves the first-order reflection principle scheme.
 \end{enumerate}
\end{theorem}

\begin{proof}
Assume \ZFCU\ holds in $V(A)$, where $V$ is the class of pure sets and $A$ is the class of atoms. Let us first prove the $\omega$-\DC\ scheme. Suppose that $R$ is a definable class relation on a definable class $X$, such that every $x\in X$ has at least one $y\in X$ such that $x\mathrel{R}y$. We want to prove that there is an $\omega$-sequence threading the relation. Fix a parameter $p$ that is sufficient for the definitions of $X$ and $R$.

We define a certain increasing sequence of cardinals $\lambda_0<\lambda_1<\lambda_2<\cdots$ by recursion. Fix a set $u$ as in lemmas \ref{Lemma.Duplication} and \ref{Lemma.Homogeneity}, so that every pair of equinumerous sets of urelements outside $u$ are automorphic over $u$. We may assume that $p$ has support in $u$. Let $\lambda_0$ be the least cardinal such that $X$ has an element inside $V_{\lambda_0}(w)$ for some set $w$ extending $u$ of size less than $\lambda_0$. If $\lambda_n$ is defined, let $\lambda_{n+1}$ be least above $\lambda_n$ such that whenever $x\in X\intersect V_{\lambda_n}(w)$, where $w$ is a set of urelements extending $u$ of size less than $\lambda_n$, and there is a $y$ with $x\mathrel{R} y$, then there is such a $y$ inside $V_{\lambda_{n+1}}(w^+)$, where $w^+$ is an extension of $w$ of size less than $\lambda_{n+1}$. This cardinal exists since all such $V_{\lambda_n}(w)$ are automorphic in $V(A)$, once the number of urelements beyond $u$ is fixed, and so there are only a set of possible $w$ to consider, and consequently a set of possible $x$ to consider. Now let $\lambda=\sup_n\lambda_n$.

By the collection axiom, we can consider the possible cardinalities of sets of urelements and then take the union of a family of such sets realizing all possible cardinalities below $\lambda$, and thereby find a set $\bar w$ of urelements extending $u$, which is at least as large as any set of urelements of size less than $\lambda$. Let $v=\Union_{w\of \bar w,\,|w|<\lambda} V_\lambda(w)$, where we range over all $w\of \bar w$ of size less than $\lambda$. Notice that if $x\in X\intersect v$, then $x$ is in some $V_{\lambda_n}(w)$ for some $w\of\bar w$ of size less than $\lambda_n$, and consequently there is $y\in V_{\lambda_{n+1}}(w')$ for some such $w'$ of size less than $\lambda_{n+1}$, and by homogeneity we may assume $w'\of\bar w$. In short, $v$ has $R$-successors for each of its elements in $X$.

It now follows easily by the axiom of choice that we may thread $R$. We simply well-order the set $v$ and then choose the least successor: fix any $x_0\in X\intersect v$, and let $x_{n+1}$ be least in $v$ such that $x_n\mathrel{R}x_{n+1}$. So $\<x_n\mid n<\omega>$ threads the relation and we have proved the $\omega$-\DC\ scheme.

Next, we observe that the $\omega$-\DC\ scheme suffices to establish every instance of the first-order reflection principle. Fix any formula $\varphi$ in the language of urelement set theory, using $\in$ and the urelement predicate $\A$, and list the subformulas $\varphi_0,\ldots,\varphi_k$. Define $u\mathrel{R} v$ for sets $u,v$ if for every $\vec x$ from $u$, if there is a witness $y$ in $V(A)$ such that $\varphi_i(\vec x,y)$, then there is such a $y\in v$. In other words, the set $v$ contains witnesses for the objects in $u$, if there are any such objects. By the collection axiom scheme, every set $u$ is related to some set $v$, and so by the principle $\omega$-\DC, there is a threading sequence $\<u_n\mid n<\omega>$, such that $u_n\mathrel{R}u_{n+1}$ for every $n$. It follows that the union set $u=\Union_n u_n$ is closed under witnesses for all the $\varphi_i$, and it follows by the Tarski-Vaught criterion that $\varphi$ and all subformulas are absolute between $u$ and $V(A)$, as desired. And we could have accommodated any particular set simply by starting with a sufficient $u_0$.
\end{proof}

\subsection{Reflection without choice}

Perhaps the reader is surprised that our proof of first-order reflection for urelement set theory made use of the axiom of choice, since in the \ZF\ context, one proves the \Levy-Montague reflection theorem without any appeal to the axiom of choice. Does the reflection principle hold in \ZFU? We don't yet know the full answer, but we can prove reflection in several weaker contexts than full \ZFCU. So let us take a small detour in this subsection to do so.

\begin{theorem}\label{Theorem.ZFU+duplication+homogeneity-implies-reflection}
 Assume \ZFU. If there is a set $u$ such that
 \begin{enumerate}
   \item every set of urelements is equinumerous with some set in $V(u)$ and
   \item homogeneity holds over $u$---all equinumerous sets of urelements disjoint from $u$ are automorphic,
 \end{enumerate}
then the first-order reflection principle holds.
\end{theorem}

In light of lemma \ref{Lemma.Homogeneity}, which shows that urelement duplication implies homogeneity, it suffices to have urelement duplication over $u$  in place of homogeneity in this theorem. The first hypothesis can be equivalently stated as: $V(u)$ realizes all cardinalities---every set is equinumerous with a set in $V(u)$. As we mentioned earlier, such a hypothesis is relevant for philosophical issues in urelement set theory as a foundation of mathematics, because $V(u)$ would have isomorphic copies of every possible mathematical structure.

Notice that the two hypotheses of the theorem hold in the theory $\ZFU+``$there are $\Ord$ many atoms'' and also in ``there are $V$ many atoms'' and ``there is a set $u$ such that there are $V(u)$ many atoms'' and these theories arise can each arise with the pure sets being an arbitrary model of \ZF, with very bad failures of the axiom of choice.

\begin{proof}
Assume \ZFU\ and fix the set $u$ for which every set of urelements is equinumerous with a set in $V(u)$ and the homogeneity principle holds over $u$. We proceed similarly to the argument of theorem \ref{Theorem.ZFCU-proves-first-order-reflection}, but avoiding the need for the axiom of choice. Fix any set $p$ and formula $\varphi$ and enumerate the subformulas $\varphi_0,\ldots,\varphi_k$ as before. We define an increasing sequence of cardinals $\lambda_0<\lambda_1<\cdots$ as before. Let $\lambda_0$ be any cardinal (and enlarge $u$ if necessary) so that any desired parameters appear in $V_{\lambda_0}(u)$. If $\lambda_n$ is defined, let $\lambda_{n+1}$ be the least cardinal above $\lambda_n$ such that if $\vec x$ is a finite tuple from $V_{\lambda_n}(w)$, where $w$ is a set of urelements equinumerous with some set in $V_{\lambda_n}(u)$ and there is a $y$ for which $\varphi_i(\vec x,y)$, then there is such a $y$ in $V_{\lambda_{n+1}}(w^+)$, where $w^+$ is an extension of $w$ that is equinumerous with some set in $V_{\lambda_{n+1}}(u)$. Such a cardinal exists, since all such $w^+$ are automorphic over $u$, once the size beyond $u$ is fixed, and so there are only a set of possible $\vec x$ to consider. Let $\lambda=\sup_n\lambda_n$.

\enlargethispage{20pt}
As before, we can consider the sets in $V_\lambda(u)$ that are the cardinalities of a set of urelements and apply collection to collect a family of such sets of urelements. The union of that collecting set will be a set $\bar w$ of urelements, whose size is at least as large as any set of urelements equinumerous with an element of $V_\lambda(u)$. Notice that the sets of urelements equinumerous with a set in $V_\lambda(u)$ are closed under finite unions, since if $w$ and $w'$ are isomorphic respectively with $x$ and $x'$ in $V_\lambda(u)$, then by using pairs we may assume $x$ and $x'$ are disjoint, and so $w\union w'$ will be equinumerous with a subset of $x\union x'$, which is in $V_\lambda(u)$. Let $v=\Union_{w\of\bar w}V_\lambda(w)$, where we allow $w\of\bar w$ that are equinumerous with a set in $V_\lambda(u)$. This set $v$ is a transitive set which is closed under witnesses for the formulas $\varphi_i$, since if $\vec x$ is in $v$ and there is a $y$ with $\varphi_i(\vec x,y)$, then since $\vec x\in V_{\lambda_n}(w)$ for some $w\of\bar w$ of size equinumerous with a set in $V_{\lambda_n}$, then we will be able to find such a $y$ in $V_{\lambda_{n+1}}(w^+)$, for a suitable size $w^+\of\bar w$, making $y\in v$. So by the Tarski-Vaught criterion, it follows that $\varphi$ is absolute between $v$ and $V(A)$, as desired.
\end{proof}\goodbreak

The following corollary is an immediate consequence.

\begin{corollary}\label{Corollary.ZFU+pure+duplication-implies-reflection}
 Assume \ZFU. If every set of urelements is equinumerous with a pure set and all sets of urelements can be duplicated over some set $u$, then the first-order reflection principle holds.
\end{corollary}

In particular, the theory $\ZFU+``$every set of urelements is well-orderable'' implies the first-order reflection principle, because all wellorderable sets are equinumerous with an ordinal, which is a pure set, and lemma \ref{Lemma.Duplication} provides a set $u$ over which sets of urelements can be duplicated. The assumption that every set of urelements is wellorderable is strictly weaker than \AC\ over \ZFU, since if one begins in a \ZF\ set-theoretic universe $V$ without \AC, but interprets $V\[A]$ for a well-orderable class $A$, such as $A=\Ord$, then every set of urelements in $V\[A]$ will be wellorderable there, but \AC\ will still fail amongst the pure sets. For example, it is relatively consistent with \ZFU\ to have a class of urelements enumerated by the ordinals, yet still $\R$ has no wellorder. For the same reason, the assertion that every set of urelements is wellorderable is strictly weaker than the assertion that the class of all urelements has a definable well order, since one cannot define a class wellorder of the urelements in that model $\<V\[A],\in,\A>$, as there will be nontrivial automorphisms of this structure swapping urelements beyond the support of any parameters used in the definition, and these would not fix the well order. Indeed, the urelements of this model can have no definable linear order.

Let us also prove that we get reflection for free whenever we construct a model of \ZFU\ by interpreting from a given model in which reflection holds. This includes every model of \ZF, and so all the interpreted models $V\[A]$ will have reflection.

\begin{theorem}\label{Theorem.Reflection-for-interpreted-models}
Assume \ZFU. If the universe $V(A)$ is interpretable in the class $V$ of pure sets, then the first-order reflection principle holds. More generally, if $V(A)$ is interpretable in a set-like manner inside a definable transitive class $W\of V(A)$ containing all pure sets and satisfying \ZFU\ plus the reflection principle, then $V(A)$ also satisfies the reflection principle.
\end{theorem}

\begin{proof}
Assume the \ZFU\ universe $V(A)$ is interpretable inside a definable transitive class $W\satisfies\ZFU$ containing the pure sets in which the reflection principle holds, and that the interpreted membership relation $\in^{V(A)}$ is set-like in $W$. Fix the definition $\psi$ for the interpretation of $V(A)$ inside $W$. For any statement $\varphi$ true in $V(A)$, consider the more complex statement, ``$\varphi$ holds in the interpreted model defined by $\psi$.'' By reflection in $W$, there a transitive set $v$ in $W$ for which $\varphi$  is absolute between $V(A)$ and the corresponding set-sized version of the universe $V(A)$ as interpreted in $v$, that is, $v$'s version of $V(A)$. We can ensure that this is a transitive set in the interpreted copy of $V(A)$ by also asking $v$ to be closed under all the $V(A)$-members of its elements, using that the interpreted relation is set-like in $W$. Thus, we found a transitive set in $V(A)$ to which $\varphi$ reflects.
\end{proof}

For example, perhaps $V$ is a model of \ZF\ in which the reals are not well orderable, say, but we can nevertheless interpret the model $V\[V]$ with an urelement for every pure set, thereby satisfying $\ZFU+``$there are $V$ many urelements.'' Theorem \ref{Theorem.Reflection-for-interpreted-models} shows that the resulting model $V\[V]$ will fulfill the first-order reflection principle. This conclusion also follows from theorem \ref{Corollary.ZFU+pure+duplication-implies-reflection}, since every set in $V\[V]$ is equinumerous with a pure set and every set of urelements can be duplicated in this model.

\section{Class theory}

We should like next to undertake the entire analysis in the context of second-order set theory, developing the urelement analogues of \Godel-Bernays set theory and Kelley-Morse set theory. These are conveniently formalized as two-sorted theories, with a first-order domain of individuals, consisting of the atoms and the sets, and a second-order domain of classes, consisting of a family of classes of individuals.

Let us begin with a quick review the second-order pure theories. A model of \Godel-Bernays set theory \GBC\ has the form $\<M,\in^M,\mathcal{M}>$, where the first-order part $\<M,\in^M>$ is a model of \ZFC, and the second-order part $\mathcal{M}\of P(M)$ is a collection of classes $X\of M$. The classes are allowed as atomic predicates into the first-order language, so that the model can assert $a\in X$ for the elements $a$ of $X$ in $M$, and we allow these class parameters into the \ZFC\ collection and separation axiom schemes. The theory \GBC\ also has the first-order class comprehension axiom, asserting that for any formula $\varphi$ involving only first-order quantifiers (over the sets) and class parameters $X_i$, the collection $X=\set{a\mid \varphi(a,X_0,\ldots,X_n)}$ forms a class in $\mathcal{M}$. The theory \GBC\ also includes the axiom of global choice, which can be equivalently expressed either as the assertion that there is a global choice function $F$ for which $F(u)\in u$ for every nonempty set $u$, as the assertion that there is a global well-order relation $<$ on the universe of sets, a principle known also as the global well-order principle, or as the $\Ord$ enumeration principle, asserting specifically that there is such an order of type $\Ord$. Although these formulations of global choice are equivalent for the pure-set purposes of \GBC, we warn the reader that they are no longer equivalent in the urelement class theories.

The theory \GBC\ is conservative over \ZFC\ for assertions about sets, because we can equip any model $M\satisfies\ZFC$ with all its first-order definable classes (with parameters), and this will satisfy all of \GBC\ except possibly the global choice principle, since in some models there is no definable global well order. But with the method of forcing we can add a generic global well order to the universe, a forcing extension adding no sets, and then take the classes that are first-order definable using this new generic class parameter. The result is a model of \GBC\ with the same sets as $M$, and so the conclusion is that any purely set-theoretic situation that can happen in \ZFC\ can also happen in \GBC, and this is an alternative equivalent way to express conservativity. It follows from the conservativity result that \GBC\ is equiconsistent with \ZFC. Despite what seems to be a very close semantic connection between \GBC\ and \ZFC, however, these theories (if consistent) are not mutually interpretable---the reason is that \GBC\ is finitely axiomatizable and consequently cannot be interpreted in \ZFC, for such an interpretation would make use of only finitely many axioms of \ZFC, but \ZFC\ proves the consistency of its finite fragments, and so such an interpretation of \GBC\ in \ZFC\ would violate the second incompleteness theorem.

Kelley-Morse set theory \KM\ strengthens \GBC\ with the second-order class comprehension axiom scheme, allowing one to define classes $\set{x\mid\varphi(x,Z)}$ by formulas $\varphi$ having second-order quantifiers ranging over the available classes of the model. This theory is not conservative over \ZFC\, since it implies that there is a satisfaction class for first-order truth, and the existence of such a truth predicate implies $\Con(\ZFC)$ and $\Con(\ZFC+\Con(\ZFC))$ and much more (see \cite{Hamkins.blog2014:KMImpliesCon(ZFC)AndMuchMore} for an elementary account). Meanwhile, if $\kappa$ is an inaccessible cardinal, then $\<V_\kappa,\in,V_{\kappa+1}>$ is a model of Kelley-Morse set theory, and so the consistency strength of \KM\ is strictly between \ZFC\ and \ZFC\ plus an inaccessible cardinal. So \KM\ ultimately is just a small step up in consistency strength from \ZFC\ or \GBC.

Gitman and Hamkins \cite{GitmanHamkins:Kelley-MorseSetTheoryAndChoicePrinciplesForClasses} (see also \cite{GitmanHamkins2016:OpenDeterminacyForClassGames, GitmanHamkinsHolySchlichtWilliams2020:The-exact-strength-of-the-class-forcing-theorem}) have observed that even Kelley-Morse set theory \KM\ does not prove the class choice principle \CC, which asserts that for every class $I$, if for every individual $i\in I$ there is a class $X$ for which $\varphi(i,X,I)$, then there is a class $X\of I\times V$ such that for every $i\in I$ we have $\varphi(i,X_i,I)$, where $X_i=\set{x\mid (i,x)\in X}$ is the $i$th section of the class $X$. For example, if for every $n$ there is a class $X$ with $\varphi(n,X)$, then \CC\ allows you to put such classes together into a single class $X\of \omega\times V$ such that $\varphi(n,X_n)$ for all $n$. This is essentially a choice principle or collection principle for classes, and even the $\omega$-indexed version $\omega$-\CC\ is not provable in \KM.

Nevertheless, every model of \KM\ admits a submodel with the same sets (but possibly fewer classes), which is a model of \KM+\CC; see \cite{GitmanHamkins:Kelley-MorseSetTheoryAndChoicePrinciplesForClasses}. It follows that \KM\ and \KM+\CC\ are mutually interpretable theories and consequently also equiconsistent. But they are not bi-interpretable, in light of \cite{Enayat2016:Variations-on-a-Visserian-theme, FreireHamkins2021:Bi-interpretation-in-weak-set-theories}. Meanwhile, the class choice principle supports important applications of Kelley-Morse set theory, and \KM+\CC\ should be seen as a natural, robust extension of \KM.

\subsection{Interpretation of Kelley-Morse in a first-order set theory}

The class choice principle supports an important bi-interpretation phenomenon of Kelley-Morse set theory, showing that \KM+\CC\ is bi-interpretable with first-order strengthenings of $\ZFCm$ with an inaccessible cardinal. The observation is due originally to Wictor Marek and Andrzej Mostowski \cite{MarekMostowski1975:On-extendibility-of-models-of-ZF-to-KM}, with perhaps an earlier form in \cite{Scott1960:A-different-kind-of-model-for-set-theory}. But  see also the excellently thorough presentation of Kameryn Williams \cite{Williams2018:dissertation}.

\begin{theorem}[Marek, Mostowski]\label{Theorem.MarekMostowski}
 The following theories are bi-interpretable.
 \begin{enumerate}
   \item Kelley-Morse set theory \KM\ with the class choice principle \CC
   \item $\ZFCm+$ there is a largest cardinal, which is inaccessible
 \end{enumerate}
\end{theorem}

\begin{proof}
To interpret the first theory in the second, assume a set-theoretic universe $\<M,\in^M>$ satisfies $\ZFCm$ and has a largest cardinal $\kappa$, which is inaccessible there. It follows that $M_\kappa=(V_\kappa)^M$ is a model of \ZFC. Inside $M$, we can consider the sets available that would form classes on $M_\kappa$, that is, the set $\mathcal{M}=\set{X\in M\mid M\satisfies X\of M_\kappa}$. The point now is that $\<M_\kappa,\in^M,\mathcal{M}>$ is a model of Kelley-Morse set theory by the same reasoning as in \ZFC\ that $\<V_\kappa,\in,V_{\kappa+1}>$ is a model of \KM\ whenever $\kappa$ is inaccessible---everything works fine just in $\ZFCm$. Note that the power set axiom holds below $\kappa$ in $M$ as a consequence of what it means for $\kappa$ to be inaccessible; and we get the class choice principle \CC\ in this \KM\ model, since instances of class choice in $\<M_\kappa,\in^M,\mathcal{M}>$ amount to instances of the first-order collection axiom in $M$, which holds because collection is part of $\ZFCm$.

For the converse interpretation, we provide the construction and sketch the argument, but refer the reader to \cite{Williams2018:dissertation} for fuller details. We begin with a model $\<V,\in,\mathcal{V}>$ of \KM+\CC\ and shall interpret a corresponding model $\<W,\varin>$ of $\ZFCm$ with a largest cardinal $\kappa$, which is inaccessible there. We employ a standard  technique of coding sets with binary relations, perhaps used more commonly when coding hereditarily countable sets with relations on $\N$, but the same idea works higher up, even with proper class codes as here, resulting in the \emph{unrolling} of a model of second-order set theory to a model of first-order set theory with sets of higher rank. Specifically, following \cite{Williams2018:dissertation}, let us say that a \emph{membership code} is a extensional class directed graph relation $E$ with a unique maximal node. (Let us agree to consider the empty relation always with $\emptyset$ as the unique maximal node, so the domain of individuals for any relation $E$ will be determined by $E$.)

As an example, for any set $x$ consider the membership relation $\<\TC(\set{x}),\in>$ on the transitive closure of the singleton $\set{x}$. This is an extensional directed graph having a unique maximal node, which is $x$ itself---the lower nodes represent the elements of $x$ and the elements of elements and so forth. Indeed, this is the principal motivating example of membership codes, for we shall think of the membership codes as representing sets completely in analogy with this case. Thus, we say that two codes are equivalent $E\sim F$, if they are isomorphic as directed graphs. For any node $y$ in the field of the relation $E$, let $E\downarrow y$ be the graph obtained by restricting to $y$ and nodes that are hereditarily below $y$---this also is a membership code, and we define $F\varin E$ if and only if $F\sim E\downarrow y$ for some $y$ with $y\mathrel{E} x$, where $x$ is the maximal node of $E$.

Williams \cite{Williams2018:dissertation} proves various welcome properties about these codes from weak fragments of \KM. For example, in $\text{GBc}^-$ one can prove that the isomorphisms of membership codes are unique when they exist, and initial partial isomorphisms of codes agree on their common domain. With \ETR, one can thus prove that any two codes admit a maximal partial initial isomorphism, and using this, one can prove extensionality for the codes.

Indeed, let $\<W,\varin>$ be the structure arising from the corresponding equivalence classes $[E]_\sim$ considered under the $\varin$ membership relation. This model is interpretable in $\<V,\in,\mathcal{V}>$, and the main observation to make is that it satisfies $\ZFCm$ with a largest cardinal, which is inaccessible there. We refer the reader to \cite{Williams2018:dissertation} for the details. The class choice principle \CC\ is used in order to verify the collection axiom in $W$, since collecting sets in $W$ amounts to collecting classes together in $\<V,\in,\mathcal{V}>$, and this is exactly what \CC\ enables. The largest cardinal $\kappa$ is simply $\Ord^V$, which is easily coded in $\mathcal{V}$ and is inaccessible there because $\Ord$ is closed under power sets and is regular with respect to classes in $\mathcal{V}$.

Finally, we observe that these interpretations provide a bi-interpretation. The original $\ZFCm$ model $\<M,\in^M>$ with a largest cardinal $\kappa$, which is inaccessible, can see how it is copied into $\<M_\kappa,\in^M,\mathcal{M}>$, since it can form the membership codes $\<\TC(\set{x}),\in^M>$ for any set $x$, and these will have size $\kappa$ and thus have isomorphic copies appearing in $\mathcal{M}$. And so the model $W$ interpreted inside $\<M_\kappa,\in^M,\mathcal{M}>$ will be a copy of $\<M,\in^M>$. And the model $\<V,\in,\mathcal{V}>$ of \KM+\CC\ can see how it arises as $W_\kappa$ inside the interpreted model $\<W,\varin>$. So this is a bi-interpretation.
\end{proof}

Since \KM\ is mutually interpretable (although not bi-interpretable) with $\KM+\CC$, it follows that \KM\ is mutually interpetable with $\ZFCm$ plus the assertion that there is a largest cardinal, which is inaccessible.

\section{Class theory with urelements}

Let us now introduce the urelement analogues of \GBC\ and \KM. Namely, \GBCU\ is the two-sorted theory for models of the form $\<M,\in^M,\A,\mathcal{M}>$, where $\<M,\in^M,\A>$ is a model of \ZFCU, allowing class parameters in the collection and separation schemes, and $\mathcal{M}$ is a collection of classes in $M$ that fulfills the first-order class comprehension axiom (in the language with $\in$ and $\A$, allowing class parameters), plus the assertion that there is a global well-order $<$ of the universe. The global well order $<$ is simply one of the classes in $\mathcal{M}$, as is the class of all urelements. The theory \KMU\ extends \GBCU\ with the second-order class comprehension axiom. The theories $\KMU$ and $\GBCU$ are formulated analogously with the enumeration predicate $\AAvec$ for the atoms rather than merely the predicate $\A$. We warn the reader that the various formulations of global choice, although equivalent in pure set theory, come apart in the urelement context (see \cite{Howard1978:Independence-results-for-class-forms-of-AC}). In particular, we won't generally be able to convert every class well order of the universe to a well order of type $\Ord$, since there may simply be too many urelements for this to be possible. We shall discuss this issue in further depth later.

In any model of \Godel-Bernays set theory $\<V,\in,\mathcal{V}>$, with any class $A\in \mathcal{V}$, we can define the interpreted urelement model $\<V\[A],\barin,\AAvec,\mathcal{V}\[A]>$ in analogy with the first-order case, defining $V\[A]$, $\barin$ and $\AAvec$ just as we did for the \ZFC\ context in section \ref{Section.Interpreting-urelements-in-ZFC}, but now also interpreting the second-order part of the models by placing the classes $\mathcal{V}\[A]=\set{B\in \mathcal{V}\mid B\of V\[A]}$. That is, the new classes are simply the old classes that happen to be subclasses of $V\[A]$.

\begin{theorem}\label{Theorem.GBC-GBCU-bi-interpretation}
 Every model $\<V,\in,\mathcal{V}>$ of \GBC\ is bi-interpretable, for every class $A\in \mathcal{V}$, with the interpreted model $\<V\[A],\barin,\AAvec,\mathcal{V}\[A]>$, which is a model of $\GBCU +``$there are $A$ many urelements.'' If the original model satisfies \KM, then the interpreted model satisfies $\KMU$, and if the original model satisfies the class choice principle \CC, then so does the interpreted model.
\end{theorem}

\begin{proof}
We know that $\<V\[A],\barin,\AAvec>$ satisfies $\ZFCU+``$there are $A$ many urelements,'' and we can allow class parameters into the collection and separation schemes. And since $V\[A]$ is first-order definable in the original model, any first-order definable class in $\<V\[A],\barin,\AAvec>$ is also definable in the original model, and so it exists as a class in $\mathcal{V}$ and will consequently be added to $\mathcal{V}\[A]$. If the original model satisfies \KM, then this argument works also with second-order definable classes, and if class choice \CC\ holds originally, it will hold in the interpreted model, since instances of class choice there amount to instances of class choice in the original model. Since $V$ can see how it is interpreted into the pure sets of $V\[A]$ via $u\mapsto\check u$, it can also transfer all the classes of $\mathcal{V}$ into the corresponding classes of pure sets $X\mapsto\hat X=\set{\check u\mid u\in X}$, and thus $\<V,\in,\mathcal{V}>$ can see how it is copied into $\<V\[A],\barin,\AAvec,\mathcal{V}\[A]>$. Conversely, the latter model can use the urelement enumeration predicate $\AAvec$ to see how it arises from the pure sets by the construction we have specified. So this is a bi-interpretation.
\end{proof}\goodbreak

The following theories are consequently all bi-interpretable with parameters:
\begin{enumerate}
  \item \GBC
  \item $\GBCU+``$there are $\omega$ many atoms''
  \item $\GBCU+``$there are $\Ord$ many atoms''
  \item $\GBCU+``$there are $V$ many atoms''
\end{enumerate}
as well as the following theories:
\begin{enumerate}
  \item \KM
  \item $\KMU+``$there are $\omega$ many atoms''
  \item $\KMU+``$there are $\Ord$ many atoms''
  \item $\KMU+``$there are $V$ many atoms''
\end{enumerate}\goodbreak

\noindent and the following:
\begin{enumerate}
  \item \KM+\CC
  \item $\KMU+\CC+``$there are $\omega$ many atoms''
  \item $\KMU+\CC+``$there are $\Ord$ many atoms''
  \item $\KMU+\CC+``$there are $V$ many atoms''
  \item $\ZFCm+``$there is a largest cardinal $\kappa$, which is inaccessible''
  \item $\ZFCU^{-}+``$there is a largest cardinal $\kappa$, which is inaccessible, and $\kappa$ many atoms''
\end{enumerate}

Note that we can take a class parameter enumerating all the urelements, if desired. In the next section we shall explain how to achieve more than $\Ord$ many urelements in models of \GBCU\ and \KMU.\goodbreak

\newpage
\section{Abundant atoms}\label{Section.Abundant-atoms}

\begin{wrapfigure}[10]{r}{.4\textwidth}\vskip-2ex\hfill
\begin{tikzpicture}[scale=.8]
\draw[fill=Orange!40] (0,0) -- (-1,4) -- (5,4) node[below left,scale=.8,align=center] {$V(A)$ } -- (4,0) -- cycle;
\draw[fill=Yellow!40] (0,0) -- (-1,4) to node[below,scale=.8] {$V$} (1,4) -- cycle;
\draw (0,3) node[scale=.6,align=center] {Pure\\ sets};
\draw[fill=Yellow] (0,0) -- (-.5,2) node[left,scale=.7] {$\kappa$} to node[below,scale=.7] {$V_\kappa$ } (.5,2) -- cycle;
\draw[fill=Orange!60] (0,0) -- (.5,2) -- (4.5,2) node[below left,scale=.7] {$V_\kappa(A)$} -- (4,0) -- cycle;
\draw[fill=Orange] (0,0) -- (.5,2) to[out=0,in=170] node[below left,scale=.7] {$H_\kappa(A)$} (4.0675,.25) -- (4,0) -- cycle;
\draw[DarkRed,dotted,line width=2pt,line cap=round, dash pattern=on 0pt off \pgflinewidth] (0,0) to node[below,scale=.45,align=center] {Atoms $A$} (4,0);
\end{tikzpicture}
\end{wrapfigure}
In urelement set theory, for any cardinal $\kappa$ we define the corresponding \emph{hereditary} class $H_\kappa(A)$ to be $(H_\kappa)^{V(A)}$, that is, the class of atoms and sets in $V(A)$ of hereditary size less than $\kappa$, the sets whose transitive closure has size less than $\kappa$. If we are working in a second-order urelement theory in the ambient universe $\<V(A),\in,\mathcal{V}>$, then we may equip $H_\kappa(A)$ with the family of classes $\mathcal{H}=\mathcal{V}\restrict H_\kappa(A)$, consisting of all the classes $X\in\mathcal{V}$ for which $X\of H_\kappa(A)$. Since every atom $a\in A$ appears in $H_\kappa(A)$, it follows that $A\of H_\kappa(A)$ and consequently also every class of urelements $B\of A$ is a class in the hereditary model $\mathcal{H}$. Thus, although $H_\kappa(A)$ has many fewer sets than $V(A)$, we nevertheless retain exactly the same classes of urelements in $\mathcal{H}$ as in $\mathcal{V}$.

\begin{theorem}\label{Theorem.GBCU-H_kappa(A)}
Assume that \GBCU\ holds in $\<V(A),\in,\mathcal{V}>$ with urelements $A$, and suppose $\kappa$ is an inaccessible cardinal in this model. Then $\<H_\kappa(A),\in,\mathcal{H}>$ is a model of \GBCU. If the ambient universe satisfies \KMU, then this hereditary model also satisfies \KMU, and similarly if \CC\ holds initially, then it holds in the hereditary model.
\end{theorem}

\begin{proof}
Consider the structure $\<H_\kappa(A),\in,\mathcal{H}>$ as a model of second-order urelement set theory. What we claim is that this is a model of \GBCU. Using the fact that $\kappa$ is inaccessible, it is straightforward to verify that $\<H_\kappa(A),\in>$ is a model of \ZFCU, even with class parameters from $\mathcal{H}$, using that this is a transitive class closed under pairing, union, and power sets, and contains as an element every subset of it of size less than $\kappa$. Since $H_\kappa(A)$ is a definable class in the ambient universe, it follows that any first-order definable class $X\of H_\kappa(A)$, allowing class parameters, will exist in $\mathcal{V}$ and hence be placed into $\mathcal{H}$. So $\<H_\kappa(A),\in,\mathcal{H}>$ will fulfill first-order class comprehension. It will satisfy the global well order principle, since $H_\kappa(A)$ admits a well order in the ambient universe and this class will therefore exist in $\mathcal{H}$. So $\<H_\kappa(A),\in,\mathcal{H}>$ will be a model of \GBCU.

Furthermore, if the ambient universe $\<V(A),\in,\mathcal{V}>$ satisfies \KMU, then classes defined by second-order definitions also will exist and hence be in $\mathcal{H}$, and so $\<H_\kappa(A),\in,\mathcal{H}>$ will be a model of \KMU. And instances of \CC\ in $\<H_\kappa(A),\in,\mathcal{H}>$ similarly reduce to instances of \CC\ in the ambient universe.
\end{proof}\goodbreak

Let us look a little more closely at this model $\<H_\kappa(A),\in,\mathcal{H}>$, for there are some interesting things to notice. The ordinals of $H_\kappa(A)$ are exactly the ordinals below $\kappa$, and so if $A$ is larger than $\kappa$ in the original universe, then because it contains every urelement in $A$, the class $H_\kappa(A)$ will not be equinumerous with $\kappa$. In other words, $\<H_\kappa(A),\in,\mathcal{H}>$ will be a model of \GBCU\ or \KMU\ in which every class has a well order, but the universe has no $\Ord$-enumeration. In particular, the so-called limitation of size principle fails in this model, even though the global choice and global well order principles hold.

Let us now assume specifically that $A$ is equinumerous with $\Ord$ in the original model and look a little more closely at what kind of classes there are in $\<H_\kappa(A),\in,\mathcal{H}>$. Since $A$ has size $\Ord$ in the original universe $\<V(A),\in,\mathcal{V}>$, it follows that for every cardinal $\delta$, including cardinals $\delta\geq\kappa$, there will be a class $B\of A$ of size $\delta$ in $V(A)$. Such a class will have been a set in $V(A)$, but is a proper class of urelements in $\<H_\kappa(A),\in,\mathcal{H}>$, although it is smaller than $A$ in size, in the sense that these two proper classes are not equinumerous.

In this model $\<H_\kappa(A),\in,\mathcal{H}>$, let us say that a class is \emph{small}, if it is strictly smaller in equinumerosity than the class $A$ of all urelements. There are many small classes in this model, since all the cardinals between $\kappa$ and $\Ord$ in the original universe $V(A)$ are realized by small proper classes $B\of A$ in $\mathcal{H}$. Indeed, the small classes $B\of A$ are exactly the classes $B\of A$ that are sets in $V(A)$. It follows from this that if $B\of A$ is a small class of urelements, then it has a power set in $V(A)$, and so we may find a class $D\of I\times B$, for some small class $I$, such that the sections $D_i=\set{b\in B\mid (i,b)\in D}$ are all different and every class $B'\of B$ is realized as $D_i$ for some $i\in I$. In other words, the model $\<H_\kappa(A),\in,\mathcal{H}>$ can see that every small class of urelements has a small power class, indexed by a class $D$ upon a small class $I$ in that way. It also follows that for every small class $I$ and every class $D\of I\times A$, such that $D_i$ is small for every $i\in I$, then $D$ itself is small. This is simply because the union of set many sets in $V(A)$ is a set, by the replacement axiom. This fact also is visible in $\<H_\kappa(A),\in,\mathcal{H}>$.\goodbreak

Let us give a name to these properties.

\begin{definition}
 The \emph{abundant atom axiom} (\AAA) is the assertion in urelement class theory that
 \begin{enumerate}
   \item The class of urelements is strictly larger than $\Ord$;
   \item every small class of urelements admits a small power class, that is,
    for every small class $B$ there is a small class $I$ and $D\of I\times B$ such that every subclass of $B$ is realized as a section $D_i$; and
   \item every small-indexed class of small classes is small, that is, if $I$ is a small class and every section $D_i$ of a class $D\of I\times A$ is small, then $D$ itself is small.
 \end{enumerate}
\end{definition}

And so we established the following improvement on theorem \ref{Theorem.GBCU-H_kappa(A)}.

\begin{theorem}\label{Theorem.GBCU-H_kappa(A)-AAA}
Assume that \GBCU\ holds with $\Ord$ many urelements $A$ and $\kappa$ is an inaccessible cardinal. Then $\<H_\kappa(A),\in,\mathcal{H}>$ is a model of \GBCU\ plus the abundant atom axiom \AAA. If the ambient universe satisfies \KMU, then this hereditary model also satisfies \KMU. If \CC\ holds originally, then it holds in the hereditary model.
\end{theorem}

Let us deepen the analysis by revealing the bi-interpretation phenomenon at the heart of the situation.

\begin{theorem}\label{Theorem.AAA-bi-interpretation-five-theories}
The following theories are bi-interpretable, with parameters.
\begin{enumerate}
  \item $\KMU+\CC+{}$the abundant atom axiom
  \item $\KM+\CC+{}$there is an inaccessible cardinal
  \item $\KMU+\CC+{}$there is an inaccessible cardinal and $\Ord$ many atoms
  \item $\ZFCm+{}$there are two inaccessible cardinals $\kappa<\lambda$, and $\lambda$ is the largest cardinal
  \item $\ZFCU^{-}+{}$there are two inaccessible cardinals $\kappa<\lambda$, where $\lambda$ is the largest cardinal, and there are $\lambda$ many atoms
\end{enumerate}
\end{theorem}

\newpage
\begin{wrapfigure}{r}{.45\textwidth}\vskip-0ex\hfill
\begin{tikzpicture}[scale=.8]
\draw[fill=Blue!15] (0,4) -- (0,5.5) -- (1,5.5) to node[below,scale=.6,align=center] {\Large $W(A)$\\[1ex] $\ZFCU^{-}$\\ $\lambda$ atoms, $\kappa,\lambda$ inacc} (4,5.5) to[out=0,in=70] (5,4);
\draw[fill=Blue!30] (-1,4) to[out=105,in=180] (0,5.5) node[below,scale=.6,align=center] {\Large $W$\\[1ex] $\ZFCm$\\ $\kappa,\lambda$ inacc} to [out=0,in=75] (1,4);
\draw[fill=Orange!40] (0,0) -- (-1,4) -- (1,4) to node[below,scale=.6,align=center] {\Large $\Vbar(A)$\\[1ex] $\KMU+\CC$\\ $\kappa$ inacc, $\Ord$ atoms}  (5,4) node[right,scale=.7] {$\lambda$} --  (4,0) -- cycle;
\draw[fill=Yellow!40] (0,0) -- (-1,4) to node[below,scale=.6,align=center] {\Large $\Vbar$\\[1ex] $\KM+\CC$\\ $\kappa$ inacc} (1,4) -- cycle;
\draw[fill=Yellow] (0,0) -- (-.5,2) to node[below,scale=.6,align=center] {\Large$V$\\ \scriptsize\KM\\ \scriptsize\CC} (.5,2) -- cycle;
\draw[fill=Orange] (0,0) -- (.5,2) to node[below,scale=.6,align=center] {\Large $V(A)$\\[1ex] $\KMU+\CC$\\ Abundant atoms} (4.5,2) node[right,scale=.7] {$\kappa$} -- (4,0) -- cycle;
\draw[DarkRed,dotted,line width=2pt,line cap=round, dash pattern=on 0pt off \pgflinewidth] (0,0) to node[below,scale=.45,align=center] {Abundant atoms $A$} (4,0);
\end{tikzpicture}
\captionsetup{style=rightside}
\caption{The bi-interpretable models}\label{Figure.Five-models}
\end{wrapfigure}
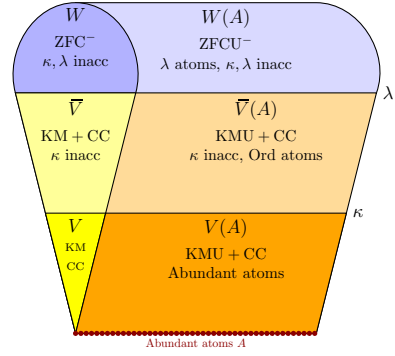
\smallskip\noindent\emph{Proof.} 
The relationships of the five bi-interpretable models are illustrated in figure \ref{Figure.Five-models}, but note that the model $V$ at lower left is not itself used directly in the interpretations. Let us start by interpreting from the bottom right of the figure to the top left, interpreting from theory (1) to theory (4). Suppose $\<V(A),\in,\mathcal{V}>$ satisfies $\KMU+\CC$ with the abundant atom axiom. We may follow the unrolling construction of theorem \ref{Theorem.MarekMostowski}, which actually works equally well in $\KMU+\CC$ as in $\KM+\CC$. That is, we consider the class membership codes $E$, allowing class relations not just on $\Ord$ or on the pure sets $V$, but on the class of urelements, and define equivalence $E\sim F$ and membership $E\varin F$. The arguments of theorem \ref{Theorem.MarekMostowski} show that the resulting universe $\<W,\varin>$ is a model of $\ZFCm$. The ordinals of $V$ become an inaccessible cardinal $\kappa$ in $W$. Because of the abundant atom axiom, however, there will be cardinals in $W$ larger than $\kappa$. The shortest well-ordering of the class of atoms will become an inaccessible cardinal $\lambda$ of $W$ strictly above $\kappa$. The cardinals below $\lambda$ in $W$ are precisely those with class membership codes that are small in $\mathcal{V}$. The abundant atom axiom ensures $\lambda$ is inaccessible in $W$, namely, a strong limit because any small class membership code will have a small class code for the power set and regular, because any small class code for a set of cardinals less than $\lambda$ will also be less than $\lambda$. Thus, $W$ is a model of $\ZFCm$ with two inaccessible cardinals $\kappa<\lambda$, where $\lambda$ is the largest cardinal.

The model $\<W,\varin>$ is in turn bi-interpretable with the extension $\<W(A),\varin>$ adjoining $\lambda$ many atoms. This bi-interpretation can be realized via the construction of $W\[\lambda]$ described in section \ref{Section.Interpreting-urelements-in-ZFC} as with theorem \ref{Theorem.V-bi-interpretable-V[A]-AAvec}, but undertaken in $\ZFCm$, where the construction works fine, producing a model of $\ZFCU^{-}$ in which $\kappa$ and $\lambda$ remain inaccessible, with $\lambda$ remaining as the largest cardinal, and the set of urelements has size $\lambda$. Thus, theories (4) and (5) are bi-interpretable.

The model $\<W,\varin>$ is bi-interpretable with $\<\Vbar,\barin,\Vcalbar>$ as in theorem \ref{Theorem.MarekMostowski}, where $\Vbar=W_\lambda$ and $\Vcalbar$ consists of the subsets of $\Vbar$ available in $W$. The model $\<\Vbar,\barin,\Vcalbar>$ has $\lambda$ as its class of ordinals, and $\kappa$ will still be inaccessible there. The model $\<\Vbar,\barin,\Vcalbar>$ is in turn bi-interpretable with $\<\Vbar(A),\barin,\Vcalbar(A)>$, adjoining $\lambda$ many atoms, because of theorem \ref{Theorem.GBC-GBCU-bi-interpretation}.

The original model $\<V(A),\in,\mathcal{V}>$ is realized as the $\<H_\kappa(A),\varin,\mathcal{H}>$ as defined in $\<W(A),\varin>$, and it can see how it is realized in that way via the interpretation to $W$ and then to $W\[A]$. So all five theories are bi-interpertable, as claimed. \hfill$\QEDbox$\medskip\goodbreak 

Let us also describe how to achieve a direct bi-interpretation of $\<V(A),\in,\mathcal{V}>$ with $\<W(A),\varin>$ and hence also with $\<\Vbar(A),\barin,\Vcalbar(A)>$. We can simply mount an urelement analogue of the unrolling construction, what might be called the \emph{urelement unrolling} of the model. Specifically, define that an \emph{urelement membership code} is a  well-founded directed graph class relation $E$ with a unique maximal node, which is extensional on all its non-minimal nodes, and where the $E$-minimal nodes are either elements of $A$ or $\emptyset$. For any set $x$ in the urelement context, the canonical membership code $\<\TC(\set{x}),\in>$ is an example, since the $\in$-minimal elements of this set will be atoms or $\emptyset$. As before, we define an equivalence relation on the codes $E\sim F$, which holds when they are isomorphic by a map fixing the atoms used as minimal elements. And we define $E\varin F$ if there is some $y\mathrel{F} x$, where $x$ is the maximal node of $F$, such that $E\sim F\downarrow y$. In the theory \KMU+\CC, one can verify that these codes work as expected to code higher sets, just as in theorem \ref{Theorem.MarekMostowski} with the pure membership codes. One shows that the isomorphisms witnessing equivalence are unique, that any two partial initial isomorphisms of codes agree on their common part, that every two codes admits a maximal partial initial isomorphism. Using the urelement analogues of the arguments made earlier, one can thus verify that the structure $\<W(A),\varin,\mathcal{W}>$, where $W(A)$ consists of the equivalence classes $[E]_\sim$ of the urelement membership codes, and $\mathcal{W}$ consists of the classes $X\of W(A)$ that are realized by a class of codes in $\mathcal{H}$. One gets the power set axiom in $\<W(A),\varin>$ below $\lambda$ using that the small classes are closed under powerclass by the abundant atom axiom. To verify collection, one uses both \CC\ as before and also the part of the abundant atom axiom asserting that the small classes are closed under small unions. The cardinal $\kappa$ arising from $\Ord^V$ is inaccessible in $W(A)$, as is the cardinal $\lambda$ that comes from the class of atoms.

If one considers only the urelement membership codes arising from small classes, the result is the model $\<\Vbar(A),\barin,\Vcalbar(A)>$, and  if one considers only the small class membership codes (with no urelements), then one realizes $\<\Vbar,\barin,\Vcalbar>$. In this way, we can directly interpret in a natural way between any two of these five models.

Finally, let us observe that we can take only the pure sets and classes of $\<V(A),\in,\mathcal{V}>$ to get the model $\<V,\in,\mathcal{V}_{\tiny pure}>$ of $\KM+\CC$ indicated in the lower left of figure \ref{Figure.Five-models}, but what we should like specifically to observe is that this model is not bi-interpretable with the other five models mentioned in the proof of theorem \ref{Theorem.AAA-bi-interpretation-five-theories}. This model has no classes of size larger than $\kappa$ and has no way to represent the cardinals of $\Vbar$ and $W$ above $\kappa$. If we were to perform the unrolling construction of this model, we would get a model of $\ZFCm$ in which $\kappa$ was the largest cardinal. In fact, it would be $H_{\kappa^+}^W=H_{\kappa^+}^{\Vbar}$, which is considerably smaller than the other models we are considering here.

\section{Second-order reflection in pure set theory}

Let us now consider the topic of reflection in models $\<V,\in,\mathcal{V}>$ of second-order set theory. Of course, we immediately get first-order reflection as with the \Levy-Montague reflection theorem---every first order statement $\varphi(X)$ that is true in $\<V,\in,\mathcal{V}>$ even with a class parameter $X\in\mathcal{V}$ is also true at some rank in the cumulative hierarchy $V_\lambda\satisfies\varphi(X\restrict V_\lambda)$.

Generalizing this, the \emph{second-order reflection principle} (evidently first considered in \cite{Bernays1976:On-the-problem-of-schemata-of-infinity}) is the scheme of assertions that every second order assertion $\varphi(X)$ true in $\<V,\in,\mathcal{V}>$ is already true in some rank-initial segment of the universe  $V_\lambda\satisfies\varphi(X\intersect V_\lambda)$.  It would be equivalent to say that every $\varphi(X)$ true in $\<V,\in,\mathcal{V}>$ is true in some transitive set $v$, that is, the statement $\varphi(X\intersect v)$ holds in the structure  $\<v,\in,P(v)>$, equipped with all its subsets. The reason is that by simply including \GBC\ as part of what $\varphi$ asserts, we can thereby assume $\<v,\in,P(v)>$ is a model of \GBC, and since the model has all subsets of $v$ this will be a model of second-order $\ZFC_2$, but by Zermelo's quasi-categoricity theorem, this implies that $v$ must be $V_\kappa$ for some inaccessible cardinal $\kappa$. So in fact, the second-order reflection principle is equivalently stated as: every second-order assertion $\varphi(X)$ true in the full universe $\<V,\in,\mathcal{V}>$ reflects to some inaccessible cardinal $\<V_\kappa,\in,V_{\kappa+1}>\satisfies\varphi(X\intersect V_\kappa)$. Note that we can equivalently state reflection as the principle that every formula $\varphi(x,X)$ with a free variable $x$ and class parameter $X$ is absolute to some inaccessible $V_\kappa$, since we can incorporate as a parameter the class of satisfying instances $x$, and reflecting that class and the fact that it is defined by $\varphi$ amounts to absoluteness of $\varphi(x,X)$ between $V_\kappa$ and $V$.

By pushing on this inaccessibility argument a little, we get a stationary proper class of inaccessible cardinals. Namely, for any class club $C\of\Ord$, we can reflect the assertion that $C$ is a class club down to some inaccessible $V_\kappa$, and from this it follows that $\kappa$ is a limit point of $C$ and hence in $C$. In other words, second-order reflection implies that every closed unbounded class of cardinals contains an inaccessible cardinal, and this is what it means to say tha $\Ord$ is Mahlo. Pushing still harder achieves even better lower bounds, as follows.

\begin{theorem}
 Assume \KM\ plus the second-order reflection principle. Then for every $n$ there are a stationary proper class of $\Pi^1_n$-indescribable cardinals.
\end{theorem}

\begin{proof}
Assume \KM\ plus second-order reflection. Fix any class club $C\of\Ord$ and any class $X\of\Ord$. By the reflection property, any $\Pi^1_n$ assertion $\varphi(X)$ reflects down to some $V_\kappa$ with $\kappa$ inaccessible and in $C$. By reflecting \emph{that} property, using a universal $\Pi^1_n$ predicate, we find such a $\kappa$ such that every subset $Y\of\kappa$ reflects every $\Pi^1_n$ property from $\<V_\kappa,\in,Y>$ to some smaller inaccessible $\<V_\gamma,\in,Y\intersect\gamma>$, with $\gamma\in C$. Thus, $\kappa$ is $\Pi^1_n$-indescribable and in $C$. So the class of them is stationary.
\end{proof}

For a quick upper bound on the strength of second-order reflection, let us observe the following; see also \cite{SolovayReinhardtKanamori1978:Strong-axioms-of-infinity-and-elementary-embeddings}.

\begin{theorem}
If $\kappa$ is a measurable cardinal, then $\<V_\kappa,\in,V_{\kappa+1}>$ is a model of $\KM+\CC$ plus the second-order reflection principle.
\end{theorem}

\begin{proof}
Suppose that $\kappa$ is a measurable cardinal and $\varphi(X)$ holds in $V_\kappa$ for some second-order assertion $\varphi$, with $X\of V_\kappa$. Let $j:V\to M$ be an elementary embedding with critical point $\kappa$. So $\varphi(j(X))$ holds in $M_{j(\kappa)}$, equipped with all its subsets in $M$. But notice that $V_\kappa$ and $V_{\kappa+1}$ both exist in $M$, and so $M$ can see that $\varphi(j(X))$ reflects from $M_{j(\kappa)}$ down to the structure $\<V_\kappa,\in,V_{\kappa+1}>$, since $j(X)\intersect V_\kappa=X$. So $M$ thinks the assertion $\varphi(j(X))$ reflects to an inaccessible cardinal below $j(\kappa)$. By elementarity, it follows in $V$ that $\varphi(X)$ must reflect from $V_\kappa$ to some inaccessible $V_\delta$ below $\kappa$, as desired.
\end{proof}

Kanamori \cite[exercise~9.18]{Kanamori2004:TheHigherInfinite2ed} shows that an $\omega$-\Erdos\ cardinal suffices for this, and this is interesting because these cardinals are consistent with $V=L$ and consequently the large cardinal upper bound on the strength of second-order reflection is comparatively low in the large cardinal hierarchy, amongst the large cardinals that can exist in $L$.

Let us also observe a curiosity, namely, that the second-order reflection principle erases the difference between \GBC\ and \KM. Indeed, one even gets global choice, as a consequence of \AC\ for sets only, as well as the principle of class choice \CC, all for free as a consequence of second-order reflection. Let \GBc\ be the theory $\GB+\AC$, that is, where we omit global choice and have just \AC\ for sets.\goodbreak

\begin{theorem}\label{Theorem.GBc+reflection=KMU+CC}
Over \Godel-Bernays set theory $\GBc$, the second-order reflection principle implies global choice, second-order class comprehension, and class choice. In short, \GBc\ plus second-order reflection is equivalent to $\KM+\CC$ plus second-order reflection. Similarly, in the urelement context $\textup{GBcU}$ plus second-order reflection is equivalent to $\KMU+\CC$ plus second-order reflection.
\end{theorem}

\begin{proof}
Assume \GBc\ and the second-order reflection principle. If there was no global well order of the universe, or if any particular axiom of \KM\ failed or any instance of \CC\ failed, then this assertion would reflect to some inaccessible $V_\kappa$. But every such model $\<V_\kappa,\in,V_{\kappa+1}>$ satisfies \KM+\CC, and global choice holds here by the axiom of choice in $V$. So those failures couldn't have occurred up in $V$ in the first place. The same idea works in the urelement context (see section \ref{Section.Second-order-reflection-with-urelements}), since the reflection set will have to have the form $H_\kappa(w)$ with an inaccessible cardinal $\kappa$ and a set of urelements $w$, and all such models similarly fulfill $\KMU+\CC$.
\end{proof}

\section{Second-order reflective cardinals and supercompactness}

We introduce a notion of second-order reflection from higher structures to smaller substructures, which will turn out to characterize the supercompact cardinals.

\begin{definition}\
\begin{enumerate}
  \item A cardinal $\kappa$ is \emph{second-order reflective}, if every second-order sentence $\varphi$ true in some structure $M$ (of any size) with $\kappa\of M$ in a language of size less than $\kappa$ is also true in a first-order elementary substructure $m\elesub M$ of size less than $\kappa$ and with $m\intersect\kappa\in\kappa$.
  \item The cardinal $\kappa$ is \emph{second-order $\lambda$-reflective}, for a cardinal $\lambda\geq\kappa$, if such reflection occurs for all models $M$ of size $\lambda$.
  \item The cardinal $\kappa$ is reflective or $\lambda$-reflective \emph{for $\Pi^1_n$ assertions}, if such reflection occurs for all sentences $\varphi$ of complexity $\Pi^1_n$.
\end{enumerate}
\end{definition}

One can equivalently drop the requirement that $m$ is elementary in $M$, requiring only that $m$ is a submodel, by Skolemizing the language to make these notions agree. It is also equivalent to consider only finite languages, rather than languages of size less than $\kappa$, since one can index the language elements by ordinals below some $\gamma<\kappa$, with $\gamma$ named as a constant, and then $m\intersect\kappa\in\kappa$ will ensure that all desired indices are in~$m$. One can accommodate individuals $a\in M$ and class parameters $X\of M$ into $\varphi$ simply by including them in the signature $\<M,a,X,\dots>$. And it is equivalent to ask for the absoluteness of a formula $\varphi(x)$ between $m$ and $M$, rather than merely a sentence, simply by including the class of satisfying instances into the language.

The $\lambda$-reflective cardinals have a natural affinity with Magidor's \cite{Magidor1971:On-the-role-of-supercompact-and-extendible-cardinals} characterization of supercompactness, namely, that $\kappa$ is supercompact if and only if for every ordinal $\lambda\geq\kappa$ there is $\alpha<\kappa$ and an elementary embedding $j:V_\alpha\to V_\lambda$ sending its critical point to $\kappa$, and more to the point, an affinity with Magidor's proof that the smallest supercompact cardinal is the least cardinal $\kappa$ such that every $\Pi^1_1$ sentence true in any structure $M$ (in a finite language) is true in a substructure of size less than $\kappa$. Despite the similarity of this hypothesis with our notion of $\lambda$-reflectivity above---it lacks only the requirement that $m\intersect\kappa\in\kappa$---our notion exactly does not appear in \cite{Magidor1971:On-the-role-of-supercompact-and-extendible-cardinals}. Nevertheless, as we prove in theorem \ref{Theorem.Reflective-supercompact} and corollary \ref{Corollary.Reflective-supercompact}, the concept of $\lambda$-reflectivity successfully characterizes supercompactness generally, not just for the least supercompact cardinal. Furthermore, theorem \ref{Theorem.Reflective-nearly-supercompact} provides a level-by-level characterization of the $\lambda$-reflective cardinals as exactly the nearly $\lambda$-supercompact cardinals. Our results can therefore be taken as a refinement of Magidor's characterizations.\goodbreak

\begin{theorem}\ \label{Theorem.Reflective-supercompact}
 \begin{enumerate}
   \item Every $\lambda$-supercompact cardinal $\kappa$ is second-order $\lambda$-reflective.
   \item Every cardinal $\kappa$ that is $2^{\lambda^{<\kappa}}\!\!$-reflective for $\Pi^1_1$ assertions is $\lambda$-supercompact.
 \end{enumerate}
\end{theorem}

\begin{proof}
(1) Suppose that $\kappa$ is $\lambda$-supercompact, and consider any structure $M=\<\lambda,R,f,\dots>$ on domain $\lambda$, with some second order statement $\varphi$ being true in $M$. Let $j:V\to N$ be a $\lambda$-supercompactness embedding. So $j(M)=\<j(\lambda),j(R),j(f),\dots>$ is a structure in $N$ in the same language, and $\varphi$ holds in $j(M)$. Consider the substructure $j\image M=\<j\image\lambda,j(R),j(f),\dots>$, which is isomorphic to $M$ by the pointwise application map $x\mapsto j(x)$. In particular, $\varphi$ is also true in $j\image M$, and this is visible in $N$, which has all the same subsets of this structure as $V$. Since $j$ is elementary, it follows also that $j\image M$ is a first-order elementary substructure of $j(M)$. It has size $\lambda$, which is less than $j(\kappa)$, and is in $N$, and furthermore, $(j\image\lambda)\intersect j(\kappa)=\kappa\in j(\kappa)$. So in $N$, we have found an elementary substructure of $j(M)$ of size less than $j(\kappa)$, whose intersection with $j(\kappa)$ is an ordinal, to which $\varphi$ reflects. By elementarity, it follows in $V$ that there must be such an elementary substructure $m$ of $M$ of size less than $\kappa$ to which $\varphi$ reflects, as desired. So $\kappa$ is second-order $\lambda$-reflective.

(2) Suppose that $\kappa$ is $2^{\lambda^{<\kappa}}$-reflective for $\Pi^1_1$ assertions, and let $M$ be the structure $\<H_{(\lambda^{<\kappa})^+},\in,\kappa,\lambda>$, which is a transitive structure of size $2^{\lambda^{<\kappa}}$, containing $\lambda$, $P_\kappa\lambda$, and every subset of $P_\kappa\lambda$, and having $\kappa$ and $\lambda$ as named constants. The question whether there is a normal fine measure on $P_\kappa\lambda$ is expressible as a second-order assertion $\varphi$ in this structure, since any such measure is determined by a subset $\mu\of M$, picking out the measure-one subsets of $P_\kappa\lambda$ and asserting that these sets constitute a normal fine measure. Normality and fineness are both first-order properties of $\mu$ over $M$, and so the assertion that $\kappa$ is not $\lambda$-supercompact is $\Pi^1_1$-expressible over $M$. So if $\kappa$ is not $\lambda$-supercompact, then this reflects to some elementary substructure $m\elesub M$ of size less than $\kappa$ with $m\intersect\kappa\in\kappa$---but note that $m$ will not be transitive. Let $s=m\intersect\lambda$, which is an element of $P_\kappa\lambda$, and let $\mu_m=\set{X\in m\mid X\of P_\kappa\lambda, s\in X}$ be the family of sets in $m$ containing $s$, the principal filter. The model $m$ can see that this measure is $\kappa$-complete, because if $\vec X=\<X_\alpha\mid\alpha<\beta>$ is a $\beta$-sequence of subsets $X_\alpha\of P_\kappa\lambda$ that $m$ thinks are all in $\mu_m$, then $\beta\in m$ and consequently  every $\alpha<\beta$ is in $m$, and so every $X_\alpha$ is in $m$. Thus, $s\in X_\alpha$ for every $\alpha<\beta$ and hence also $s\in\Intersect_{\alpha<\beta}X_\alpha$, which is $\Intersect\vec X$ as computed in $m$. Since every ordinal $\alpha<\lambda$ in $m$ is in $s$, it follows that $m$ will think that $\mu_m$ is fine. And since every element of $s$ is in $m$, it follows that $m$ will think that every regressive function is constant on a $\mu_m$-measure one set. So $m$ thinks that $\kappa$ is $\lambda$-supercompact, contrary to the reflection assumption.
\end{proof}\goodbreak

\begin{corollary}\label{Corollary.Reflective-supercompact}
The following are equivalent:
    \begin{enumerate}
     \item $\kappa$ is second-order reflective.
      \item $\kappa$ is reflective for $\Pi^1_1$ assertions.
      \item $\kappa$ is a supercompact cardinal.
    \end{enumerate}
\end{corollary}

\begin{proof}
Theorem \ref{Theorem.Reflective-supercompact} shows that supercompactness implies reflectivity level-by-level, and we only need reflectivity for $\Pi^1_1$ assertions at $2^{\lambda^{<\kappa}}$ to get $\lambda$-supercompactness.
\end{proof}

The second-order reflectivity hierarchy is therefore simply a different, finer manner of stratifying the supercompactness hierarchy. Indeed, let us mount a somewhat more delicate analysis, which will enable us to provide an exact level-by-level equivalence using the notion of near supercompactness. Specifically, Schanker \cite{Schanker2011:Dissertation, Schanker2011:WeaklyMeasurableCardinals, Schanker2013:PartialNearSupercompactness}, in dissertation work undertaken with the first author of this paper, defines that a cardinal $\kappa$ is \emph{nearly $\lambda$-supercompact}, if for every $A\of\lambda$ there is a transitive model $M\satisfies\ZFCm$, closed under $\ltkappa$-sequences, with $\lambda,A\in M$, and there is an elementary embedding $j:M\to N$ with critical point $\kappa$ for which $\lambda<j(\kappa)$ and $j\image\lambda\in N$. He proves an abundance of equivalent characterizations in \cite[theorem~2.1.3]{Schanker2011:Dissertation},
one of which is the normal filter property: $\kappa$ is nearly $\lambda$-supercompact if and only if for every collection $M$ having at most $\lambda$ many subsets of $P_\kappa\lambda$ and at most $\lambda$ many functions $f:P_\kappa\lambda\to\lambda$, there is a $\kappa$-complete fine filter on $P_\kappa\lambda$ measuring every subset in $M$ that is also $M$-normal, meaning that every regressive function $f$ in $M$ is constant on a measure one set.

\begin{theorem}\label{Theorem.Reflective-nearly-supercompact}
Assume $\lambda=\lambda^{<\kappa}$. Then the following are equivalent:
 \begin{enumerate}
     \item $\kappa$ is $\lambda$-reflective for $\Pi^1_1$ assertions
     \item $\kappa$ is nearly $\lambda$-supercompact
 \end{enumerate}
\end{theorem}

Every nearly $\lambda$-supercompact cardinal $\kappa$ is also nearly $\lambda^{<\kappa}$-supercompact, by \cite[lemma~2.1.5]{Schanker2011:Dissertation}, and so we could have alternatively stated this theorem as the claim: a cardinal $\kappa$ is $\lambda^{<\kappa}$-reflective for $\Pi^1_1$ if and only if it is nearly $\lambda$-supercompact.

\begin{proof}
Suppose that $\kappa$ is nearly $\lambda$-supercompact and $\lambda=\lambda^{<\kappa}$. Consider any first-order structure $\Lambda=\<\lambda,R,f,\dots>$ with a signature of size less than $\kappa$, satisfying a $\Pi^1_1$ sentence $\varphi$. We may place this structure into a transitive model $M\elesub H_{\lambda^+}$ of size $\lambda$, closed under $\ltkappa$-sequences. So there is an elementary embedding $j:M\to N$ with critical point $\kappa$ and $j\image\lambda\in N$. It follows that $j\image\Lambda$ is in $N$, isomorphic to $\Lambda$, and a first-order elementary substructure of $j(\Lambda)$. Since $\varphi$ is $\Pi^1_1$ and is actually true in $\Lambda$, it will also be true of $j\image\Lambda$ in $N$. (This step of the argument would not necessarily work for more complicated assertions, since $N$ may not have all subsets of $j\image\Lambda$.) So $N$ thinks that $\varphi$ reflects from $j(\Lambda)$ to a small elementary substructure, whose intersection with $j(\kappa)$ is an ordinal. By elementarity, $M$ thinks the same of $\Lambda$, and since $M\elesub H_{\lambda^+}$, this is also true in $V$. So $\kappa$ is $\lambda$-reflective for $\Pi^1_1$ assertions.

Conversely, suppose now that $\kappa$ is $\lambda$-reflective for $\Pi^1_1$ assertions, and consider any family of at most $\lambda$ many subsets of $P_\kappa\lambda$ and $\lambda$ many functions $f:P_\kappa\lambda\to\lambda$. Place these into a transitive structure $M\elesub H_{\lambda^+}$ of size $\lambda$ and closed under $\ltkappa$-sequences. The nonexistence of an $M$-normal fine $\kappa$-complete filter measuring every set in $M$ is expressed by a $\Pi^1_1$ sentence over $M$. So if indeed there were no such filter, this would reflect to some elementary substructure $m\elesub M$ of size less than $\kappa$ with $m\intersect\kappa\in\kappa$. As in the proof of theorem \ref{Theorem.Reflective-supercompact}, we may use the set $s=m\intersect \lambda$ to define a filter $\mu_m=\set{X\in m\mid X\of P_\kappa\lambda, s\in X}$, and this will be fine, $m$-normal, and $\kappa$-complete. So $m$ thinks that there is an $m$-normal fine $\kappa$-complete filter, contrary to the reflection. So $\kappa$ must have had the normal fine filter property originally and thus it is nearly $\lambda$-supercompact.
\end{proof}

Theorem \ref{Theorem.Reflective-nearly-supercompact} thus provides a more refined analysis than theorem \ref{Theorem.Reflective-supercompact} in the characterization of the cardinals that are $\lambda$-reflective for $\Pi^1_1$ assertions, which are exactly the nearly $\lambda$-supercompact cardinals. This theorem is part of a constellation of closely related, similar results in the research literature, including \cite[theorems~3.5,4.7 ]{Carr1985:Pxd-generalizations-of-weak-compactness}, \cite[theorem~1.4]{Cody2020:Characterizations-of-the-weakly-compact-ideal-on-Pkappalambda}, and 
\cite[lemma~2.8]{HayutMagidor2022:Subcompact-cardinals-type-omission-and-ladder-systems}.

In the \KM\ context, we say that $\kappa$ is \emph{second-order $X$-reflective}, for any class $X\fo\kappa$, if any such assertion $\varphi$ true in a class structure $M$ on domain $X$ is also true in a set structure $m\elesub M$ of size less than $\kappa$ with $m\intersect\kappa\in\kappa$. And similarly with \emph{$X$-reflective for $\Pi^1_1$ assertions}, restricting the complexity of $\varphi$. If $\kappa$ is $\Ord$-reflective for $\Pi^1_1$, then it is $\lambda$-reflective for $\Pi^1_1$ for all $\lambda\geq\kappa$, and so by theorem \ref{Theorem.Reflective-supercompact}, it follows that $\kappa$ is supercompact. In this sense, $\Ord$-reflectivity is a strengthening of supercompactness.

\section{Second-order reflection with urelements}\label{Section.Second-order-reflection-with-urelements}

Let us now analyze the situation of second-order reflection in the urelement context. A model $\<V(A),\in,\mathcal{V}>$ of second-order urelement set theory \GBCU\ satisfies the \emph{second-order reflection principle} if every second-order statement $\varphi(X)$ true in that model of some class parameter $X\in\mathcal{V}$ reflects to some some transitive set $v\in V(A)$, meaning that $\<v,\in,P(v)>\satisfies\varphi(X\intersect v)$. By including $\GBCU$ as part of $\varphi$, we may assume $\<v,\in,P(v)>\satisfies\GBCU$, and from this it follows that $v$ must have the form $H_\kappa(w)$ for some inaccessible cardinal $\kappa$ and set of urelements $w$, because $v$ will be closed under power sets and will have to contain all the size-less-than-$\kappa$ subsets of itself, where $\kappa=v\intersect\Ord$, which will consequently be inaccessible. (Note that when $w$ has size $\kappa$ or larger, this is not the same as $V_\kappa(w)$, since $w$ itself has rank $1$ and would appear as a set, but in $H_\kappa(w)$ it would be a proper class.) The second author proved in \cite{Yao2022:Reflection-principles-and-second-order-choice-principles-with-urelements} that \KMU\ with at most $\Ord$ many urelements and second-order reflection is bi-interpretable with \KM\ plus second-order reflection. The main contribution of this article, in contrast, will be to observe the dramatic increase in large-cardinal strength of second-order reflection when there are more than $\Ord$ many urelements.  

Let us begin with the observation that from supercompactness, we can produce models of urelement set theory with more than $\Ord$ many urelements, yet with second-order reflection. This observation is one of the core ideas and main results proved by the second-author in  \cite{Yao2022:Reflection-principles-and-second-order-choice-principles-with-urelements, Yao2023:Dissertation}, showing that if $\kappa$ is $\kappa^+$-supercompact, then there is a model of \KMU\ with strictly more than $\Ord$ many urelements and the second-order reflection principle. Here, we generalize the observation to higher levels of supercompactness.

\begin{theorem}\label{Theorem.Supercompact-to-reflection}
Assume $V$ satisfies $\ZFC+\kappa$ is $\lambda$-supercompact, where $\lambda>\kappa$. In the interpreted model $V\[\lambda]$ of \ZFCU\ with $\lambda$ many urelements $A$, the hereditary model $\<H_\kappa(A),\in,\mathcal{H}>$ is a model of $\KMU+\CC$ plus second-order reflection with more than $\Ord$ many urelements, indeed, $\lambda$ many.
\end{theorem}

\begin{proof}
Let $V(A)$ be the extension of $V$ by adding $\lambda$ many urelements, isomorphic to the model $V\[\lambda]$ as interpreted inside $V$. We consider $H_\kappa(A)$ as defined in $V(A)$, with its isomorphic copy $H_\kappa\[\lambda]$ interpreted in $V$, taking $\mathcal{H}$ as the power set of $H_\kappa(A)$ in $V(A)$. (We can build $H_\kappa\[\lambda]$ directly in $V$ by starting with the urelement objects $\<0,\alpha>$ for $\alpha<\lambda$ and then closing under the operation $y\of H_\kappa\[\lambda]\implies \<1,y>\in H_\kappa\[\lambda]$ for $y$ of size less than $\kappa$.) This is a model of $\KMU+\CC$ by
theorem \ref{Theorem.GBCU-H_kappa(A)}, and since $\lambda>\kappa$, the model thinks the class of urelements is strictly larger than $\Ord$. By replacing $\lambda$ with $\lambda^{<\kappa}$, we may assume $\lambda=\lambda^{<\kappa}$. The structure $\<H_\kappa(A),\in,\mathcal{H}>$, has size $\lambda$, and so by theorem \ref{Theorem.Reflective-supercompact}, every second-order statement $\varphi(X)$ true in that structure reflects to some first-order elementary substructure $m$ of size less than $\kappa$, whose intersection with $\kappa$ is an ordinal $m\intersect\kappa\in m$. This implies $m$ is transitive, since if $x\in m$, then $m$ can enumerate it in some order type $\gamma$, which must be less than $\kappa$, but since $m\intersect\kappa\in\kappa$, it follows that $\gamma\of m$ and so every element of $x$ is also in $m$. So it is transitive, and so we have verified second-order reflection in this model.
\end{proof}

Next, we improve on the hypothesis by observing that the proof can be undertaken with the weaker assumption that $\kappa$ is merely nearly $\lambda$-supercompact, instead of $\lambda$-supercompact. We simply place $H_\kappa\[\lambda]$ into a transitive $M\elesub H_{\lambda^+}$ of size $\lambda$, closed under $\ltkappa$-sequences, and then apply theorem \ref{Theorem.Reflective-nearly-supercompact}. This provides:

\begin{theorem}\label{Theorem.Nearly-supercompact-KMU+reflection}
Assume $V$ satisfies $\ZFC+\kappa$ is nearly $\lambda$-supercompact, where $\lambda>\kappa$. In the interpreted model $V\[\lambda]$ of \ZFCU\ with $\lambda$ many urelements $A$, the hereditary model $\<H_\kappa(A),\in,\mathcal{H}>$ is a model of $\KMU+\CC$ plus $\Pi^1_1$ reflection with more than $\Ord$ many urelements, indeed, $\lambda$ many.
\end{theorem}


A consequence of theorem \ref{Theorem.Nearly-supercompact-KMU+reflection} is that to produce a model of $\KMU+\CC$ with strictly more than $\Ord$ many atoms and second-order reflection, it suffices to have a nearly $\kappa^+$-supercompact cardinal, which is strictly weaker than $\kappa^+$-supercompact, although this hypothesis remains very strong in large cardinal terms---for example, by \cite{Schanker2011:Dissertation} it implies $\AD^{L(\R)}$.

\section{Second-order reflection with abundant atoms implies supercompact}

We come now finally to the main result of this article, where we aim to prove conversely that supercompactness is required for second-order reflection in the context of urelement set theory when there are abundant atoms, even just for $\Pi^1_1$ reflection.

\begin{maintheorem}\label{Theorem.AAA+RP2-bi-interpretation-five-theories}
The following theories are bi-interpretable.
\begin{enumerate}[start=0]
  \item $\GBc+{}$abundant atom axiom $+$ second-order reflection
  \item $\KMU+\CC+{}$abundant atom axiom $+$ second-order reflection
  \item $\KM+\CC+\kappa\text{ is supercompact and second-order }\Ord\text{-reflective}$
  \item $\KMU+\CC+\Ord\text{ many atoms}+\kappa$ is supercompact and second-order $\Ord \text{-reflective}$
  \item $\ZFCm+\kappa$ is $\ltlambda$-supercompact and (moreover) second-order $\lambda$-reflective, where $\lambda$ is the largest cardinal and inaccessible
  \item $\ZFCU^{-}+\lambda$ many atoms${}+\kappa$ is $\ltlambda$-supercompact and (moreover) second-order \text{$\lambda$-reflective}, where $\lambda$ is the largest cardinal and inaccessible
\end{enumerate}
\end{maintheorem}

\newpage
\begin{wrapfigure}[14]{r}{.4\textwidth}\hfill
\begin{tikzpicture}[scale=.8]
\draw[fill=Blue!15] (0,4) -- (0,5.5) -- (1,5.5) to node[below,scale=.45,align=center] {\Large $W(A)$\\[1ex] $\ZFCU^{-}$\\ $\lambda$ inaccessible, $\lambda$ atoms\\ $\kappa$ is $\ltlambda$-supercompact,\\ $\lambda$-reflective} (4,5.5) to[out=0,in=70] (5,4);
\draw[fill=Blue!30] (-1,4) to[out=105,in=180] (0,5.5) node[below,scale=.45,align=center] {\Large $W$\\[1ex] $\ZFCm$\\ $\lambda$ inacc\\ $\kappa$ is $\ltlambda$-supercompact,\\ $\lambda$-reflective} to [out=0,in=75] (1,4);
\draw[fill=Orange!50] (0,0) -- (-1,4) -- (1,4) to node[below,scale=.5,align=center] {\Large $\Vbar(A)$\\[2ex] $\KMU+\CC$\\ $\Ord$ many atoms\\ $\kappa$ supercompact, $\Ord$-reflective}  (5,4) node[right,scale=.7] {$\lambda$} --  (4,0) -- cycle;
\draw[fill=Yellow!50] (0,0) -- (-1,4) to node[below,scale=.45,align=center] {\Large $\Vbar$\\[1ex] $\KM+\CC$\\ $\kappa$ supercompact,\\ $\Ord$-reflective} (1,4) -- cycle;
\draw[fill=Yellow] (0,0) -- (-.5,2) to node[below,scale=.5,align=center] {\Large$V$\\ \scriptsize\KM\\ \scriptsize\CC} (.5,2) -- cycle;
\draw[fill=Orange] (0,0) -- (.5,2) to node[below,scale=.5,align=center] {\Large $V(A)$\\[1ex] $\KMU+\CC$\\ abundant atoms\\ second-order reflection} (4.5,2) node[right,scale=.7] {$\kappa$} -- (4,0) -- cycle;
\draw[DarkRed,dotted,line width=2pt,line cap=round, dash pattern=on 0pt off \pgflinewidth] (0,0) to node[below,scale=.45,align=center] {abundant atoms $A$} (4,0);
\end{tikzpicture}
\captionsetup{style=rightside}
\caption{Supercompactness bi-interpretation}\label{Figure.Supercompact-five-models}
\end{wrapfigure}
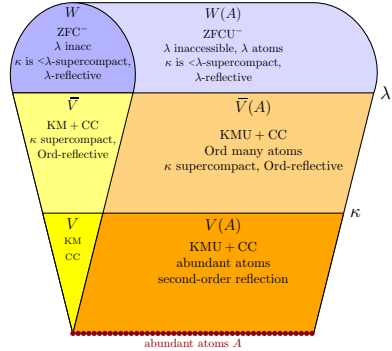
\smallskip\noindent\emph{Proof.} 
Several of the theories are redundantly stated in order to highlight the presence of supercompactness; namely, as we mentioned earlier every second-order $\Ord$-reflective cardinal is automatically supercompact, and similarly every second-order $\lambda$-reflective cardinal for an inaccessible cardinal $\lambda$ is also $\ltlambda$-supercompact by theorem \ref{Theorem.Reflective-supercompact}. Furthermore, theories (0) and (1) are provably identical in light of theorem \ref{Theorem.GBc+reflection=KMU+CC}. For the remaining bi-interpretations, we shall use the same interpretations as in theorem \ref{Theorem.AAA-bi-interpretation-five-theories}, while observing how the stronger reflection properties here manifest in the different models.

Let us start at the lower right of figure \ref{Figure.Supercompact-five-models} with the model $\<V(A),\in,\mathcal{V}>$ of $\KM+\CC$ with the abundant atom axiom and second-order reflection. We observe first that the second-order reflection property ensures that $\kappa$ is second-order $\lambda$-reflective in $W$. Notice that the concept of $\lambda$-reflective makes sense as a scheme in a model of $\ZFCm$, since it refers only to objects of size $\lambda$. But any structure $M$ in $W$ of size $\lambda$ is coded by a class in $\<V(A),\in,\mathcal{V}>$, which has the same subclasses there as the subsets available in $W$. We can equip this structure with a first-order truth predicate, and by the second-order reflection principle every second-order assertion in that expanded structure will reflect to a transitive set in $V(A)$, which because of the truth predicate will also be a first-order elementary substructure in $W$ of size less than $\kappa$. The transitivity of the substructure in $V(A)$ translates to transitivity-below-$\kappa$ in $W$ (urelements in substructure in $V(A)$ will give rise possibly to nontransitivity above $\kappa$). So $\kappa$ is $\lambda$-reflective in $W$. From this it follows that $\kappa$ is $\ltlambda$-supercompact in $W$, which implies both that $\kappa$ is fully supercompact in $\Vbar$, which is simply $W_\lambda$, and also that $\kappa$ is $\Ord$-reflective in $\<\Vbar,\barin,\Vcalbar>$, which itself translates back to $\lambda$-reflectivity in $W$. The supercompactness of $\kappa$ carries over through the bi-interpretations to $\Vbar(A)$ and to $W(A)$, and conversely, and so we've got all the theories just as we want them.
$\QEDbox$\medskip\goodbreak 

The second-order reflection principle, we claim, does not necessarily hold in $\<\Vbar,\barin,\Vcalbar>$, since $\lambda$ could be the next inaccessible above $\kappa$, and this would violate second-order reflection, which implies that the inaccessible cardinals (and the $\Pi^1_n$-indescribable cardinals, etc.) form a stationary proper class. This situation arises if $\kappa$ is $\lambda$-supercompact for the next inaccessible cardinal $\lambda$ and we form $H_\kappa\[\lambda]$ as in theorem \ref{Theorem.Supercompact-to-reflection}. This would be a model of theory (1), but the corresponding $\Vbar$ would have no inaccessible cardinals above $\kappa$ and hence not fulfill second-order reflection.\goodbreak

We get an analogous version of the theorem for $\Pi^1_1$-reflectivity as follows.

\begin{maintheorem}[$\Pi^1_1$ variation]\label{Theorem.AAA+Pi11-bi-interpretation-five-theories}
The following theories are bi-interpretable.
\begin{enumerate}
  \item $\KMU+\CC+{}$abundant atom axiom $+\Pi^1_1$ reflection
  \item $\KM+\CC+\kappa\text{ is supercompact and }\Ord\text{-reflective for }\Pi^1_1$
  \item $\KMU+\CC+\Ord\text{ many atoms}+\kappa\text{ supercompact and }\Ord\text{-reflective for }\Pi^1_1$
  \item $\ZFCm+\kappa$ is $\ltlambda$-supercompact and (moreover) nearly $\lambda$-supercompact, where $\lambda>\kappa$ is the largest cardinal and inaccessible
  \item $\ZFCU^{-}+\lambda\text{ many atoms}+\kappa$ is $\ltlambda$-supercompact and (moreover) nearly $\lambda$-supercompact, where $\lambda>\kappa$ is the largest cardinal and inaccessible
\end{enumerate}
\end{maintheorem}

The point is that $\Pi^1_1$-reflection in the original model carries over to $\Pi^1_1$ reflection in the other models (and conversely), ensuring in $W$ that $\kappa$ is $\lambda$-reflective for $\Pi^1_1$ assertions, which is equivalent to it being nearly $\lambda$-supercompact by theorem \ref{Theorem.Reflective-nearly-supercompact}.

Let us finally also address the central question left open in \cite{Yao2022:Reflection-principles-and-second-order-choice-principles-with-urelements}, concerning the strength of \KMU\ with second-order reflection and ``more than $\Ord$ many atoms.'' Yao had used a $\kappa^+$-supercompact cardinal $\kappa$ in order to produce a model of this theory, and the question is whether supercompactness is required, and specifically whether the large cardinal requirements will exceed the large cardinal notions consistent with $V=L$.

First, we know by theorem \ref{Theorem.Nearly-supercompact-KMU+reflection} that for $\Pi^1_1$ reflection it suffices to have merely a nearly $\kappa^+$-supercompact cardinal $\kappa$, and since this is strictly weaker than $\kappa^+$-supercompactness, this shows that full $\kappa^+$-supercompactness is not required for this amount of reflection.

Meanwhile, second, the hypothesis that there is a nearly $\kappa^+$-supercompact cardinal $\kappa$ is still quite strong---by \cite{Schanker2011:Dissertation} this hypothesis (with $2^\kappa=\kappa^+$) implies $\AD^{L(\R)}$ over \ZFC, which brings us to a strong realm of the large cardinal hierarchy. We can prove the following equiconsistency:\goodbreak

\begin{theorem}\label{Theorem.KMU+-strength}
The following theories are mutually interpretable, hence also equiconsistent
\begin{enumerate}
  \item $\KMU+\CC+\textup{more than }\Ord\textup{ many atoms}+\Pi^1_1\textup{-reflection}$
  \item $\ZFCm+\kappa$ is nearly $\kappa^+$-supercompact
\end{enumerate}
\end{theorem}

\begin{proof}
Suppose that $\<V(A),\in,\mathcal{V}>$ is a model of the first theory. Let $\<W,\varin>$ be the model of $\ZFCm$ obtained by the unrolling construction. Let $\kappa$ be the cardinal  in $W$ that results from a membership code coding the class $\Ord^V$ itself. This will be an inaccessible cardinal in $W$, because $V(A)$ thinks that $\Ord$ is closed under power sets and regular with respect to class in $\mathcal{V}$. Since $A$ has size larger than $\Ord$, however, there will cardinals in $W$ strictly larger than $\kappa$, so $\lambda=\kappa^+$ will exist in $W$ (and perhaps many more cardinals, depending on how many proper class cardinalities there are in $\mathcal{V}$). The argument of theorem \ref{Theorem.AAA+RP2-bi-interpretation-five-theories} shows that $\kappa$ is reflective for $\Pi^1_1$ assertions in $W$ and hence nearly $\lambda$-supercompact there, providing an interpreted model of the first theory.

Conversely, if $\<W,\in>$ is a model of the second theory, where $\kappa$ is nearly $\kappa^+$-supercompact, then we may form the interpreted model $W\[\kappa^+]$, view it as adjoining a set of $\kappa^+$ many urelements $W(A)$, and then interpret the model $\<H_\kappa(A),\in,\mathcal{H}>$, where $\mathcal{H}$ consists of the subsets $X\of H_\kappa(A)$ that are available in $W(A)$. This is a model of $\KMU+\CC$ with $\Ord=\kappa$ and $\kappa^+$ many atoms. The fact that $\kappa$ is nearly $\kappa^+$-supercompact in $W$ implies the reflection principle for $\Pi^1_1$ assertions in the hereditary model, providing a model of the first theory.
\end{proof}

We could achieve a bi-interpretation if we specifically add to the first theory that there are precisely $\Ord^+$ many atoms, meaning that the class of atoms is not equinumerous with $\Ord$, but every proper class is bijective with $\Ord$ or with the class of all atoms.

Although the hypothesis that $\kappa$ is nearly $\kappa^+$-supercompact is strong in \ZFC, we are unsure how much of this large cardinal strength can be established in $\ZFCm$. We conjecture, however, that even in the $\ZFCm$ context this hypothesis will imply that there is an inner model of \ZFC\ with a measurable cardinal and probably much more; we leave this problem for another time.

\section{Concluding philosophical remarks}

How shall we consider urelements in set theory? 
Historically, set theory in the main has largely abandoned urelements in its fundamental theories, replacing the early urelement set theories with pure set theories such as \ZF\ and \ZFC, a development that one can easily explain on structuralist grounds (a point made also in \cite[\S 8.4]{Hamkins2021:Lectures-on-the-philosophy-of-mathematics}). Namely, since all the mathematical structures that early set theorists wanted to build with urelements, such as number systems or geometric spaces, have found isomorphic copies within the pure sets, and since structuralists do not care which particular objects will be used to constitute a mathematical structure, considering it only as invariant under isomorphism, the urelements are seen as inessential. Indeed, a mathematician who favors urelements on the grounds that some mathematical objects at bottom are not sets and should not be represented with sets can be seen as preoccupied with a misguided anti-structuralist concern, one that, according to structuralism, is irrelevant for mathematical advance.

Furthermore, the bi-interpretability of many natural formulations of urelement set theory with corresponding pure set theories, as in theorems \ref{Theorem.V-bi-interpretable-V[A]-AAvec} and \ref{Theorem.GBC-GBCU-bi-interpretation}, is itself an explanation of precisely how those particular urelement conceptions can be dispensed with in the foundations of mathematics---any mathematical structure to be found in the urelement set theories can be found via the bi-interpretation also in the corresponding pure set theories. And theorems \ref{Theorem.AAA-bi-interpretation-five-theories} and \ref{Theorem.AAA+RP2-bi-interpretation-five-theories} show that this remains true even when one adds abundant urelements and second-order reflection.

If urelement set theories are to play an indispensible role in the foundations of mathematics, therefore, it must not be with those theories, but rather with urelement set theories that are not bi-interpretable with a pure set theory and perhaps not interpretable at all in any pure set theory. But in this case, it would seem that the urelement set theories must involve much stranger sets of urelements, neither well-orderable nor even equinumerous with any pure set. The mathematical structures built on such domains will not be isomorphic with any structure to be found amongst the pure sets. But what are these strange urelements that give rise to these weird sets? One wants an explanation for why we should need or expect to find such sets in the foundations of mathematics. What mathematical structures will they represent?

In second-order urelement set theory, the abundant classes of urelements are perhaps second-order instances of such strange classes, since they are not equinumerous with any class of pure sets. Indeed, the non-equivalence of the class well-order principle with the $\Ord$-enumeration of the universe (equivalently expressed by the so-called limitation of size principle) opens to the door to the possibility of strictly more than $\Ord$ many atoms, even when they are well-orderable, with the abundant atom axiom simply carrying this to an extreme. Such urelement universes are short and fat---overflowing in width with abundant atoms but limited in height with comparatively fewer ordinals. Precisely because of this, perhaps the main lesson of theorem \ref{Theorem.AAA-bi-interpretation-five-theories} is that the presence of so many urelements simply indicates that one hasn't continued the cumulative hierarchy high enough. One should add more ordinals on top of the universe, continuing the rank hierarchy to higher levels, and the abundant atoms available tell you exactly how to do this. Indeed, the unrolling constructions of $\Vbar$ and $W$ from the model $\<V(A),\in,\mathcal{V}>$ are exactly implementing this idea of continuing to build the ordinals and rank levels higher. The result is a model $\Vbar(A)$ in which the same class of atoms now has type $\Ord$, which is to say, the new higher class of ordinals $\Ord^{\Vbar}$, or even forms a set as in $W(A)$, with more ordinals still. From this perspective, the abundant atom axiom reveals the set-theoretic universe as an unfinished project---it should have been built taller. The weaker hypothesis of having a well ordered class of strictly more than $\Ord$ many urelements is similarly unnatural, a sign that one simply has not continued the cumulative hierarchy high enough.

Where does this leave us? If urelements are to be well-orderable, then we should have at most $\Ord$ of them, or else it is a sign that we have not properly built the cumulative hierarchy of pure sets; but if we have at most $\Ord$ of them, then we don't need them at all, since having $\Ord$ many atoms (or fewer) is bi-interpretable with the pure set theories. So if urelements are to play a critical role in the foundations of mathematics, it must be that they are not wellorderable, and furthermore there must be sets of them that are not equinumerous with any pure set.

\printbibliography

\end{document}